\documentclass[reqno,oneside,11pt]{amsart}

\usepackage[english]{babel}
\usepackage[utf8]{inputenc}
\usepackage{amsfonts,amsmath,amsthm}
\usepackage{stmaryrd}
\usepackage{thmtools,thm-restate}
\usepackage[T1]{fontenc}
\usepackage{enumitem}
\usepackage{varioref}
\usepackage[all]{xy}
\usepackage{units}
\usepackage{hyperref}
\usepackage{graphicx}
\usepackage{amssymb}
\usepackage{mathabx}
\usepackage{times,mathptm, mathtools}
\usepackage{etoolbox}
\usepackage{tikz-cd}
\usepackage{mathrsfs}
\usepackage{mathdots}
\usepackage[left=2.4cm, right=2.4cm]{geometry}

\docsvlist{A,B,C,D,E,F,G,H,I,J,K,L,M,N,O,P,Q,R,S,T,U,V,W,X,Y,Z}

\docsvlist{A,B,C,D,E,F,G,H,I,J,K,L,M,N,O,P,Q,R,S,T,U,V,W,X,Y,Z}

\docsvlist{A,B,C,D,E,F,G,H,I,J,K,L,M,N,O,P,Q,R,S,T,U,V,W,X,Y,Z}

\docsvlist{A,B,C,D,E,F,G,H,I,J,K,L,M,N,O,P,Q,R,S,T,U,V,W,X,Y,Z}

\docsvlist{A,B,C,D,E,F,G,H,I,J,K,L,M,N,O,P,Q,R,S,T,U,V,W,X,Y,Z}

\renewcommand{\hat}{\widehat}
\newcommand{\R}{\mathbf{R}}
\newcommand{\C}{\mathbf{C}}
\newcommand{\Q}{\mathbf{Q}}
\newcommand{\Z}{\mathbf{Z}}

\newcommand{\A}{\mathbf{A}}

\renewcommand{\P}{\mathbf{P}}

\newcommand{\G}{\mathbb G}

\renewcommand{\div}{\operatorname{div}}
\DeclareMathOperator{\Pic}{Pic}
\DeclareMathOperator{\Div}{Div}
\DeclareMathOperator{\QAlb}{QAlb}
\DeclareMathOperator{\Alb}{Alb}

\DeclareMathOperator{\ord}{ord}
\newcommand{\OO}{\mathcal O}
\DeclareMathOperator{\id}{id}

\DeclareMathOperator{\pr}{pr}
\DeclareMathOperator{\Supp}{Supp}

\DeclareMathOperator{\Nef}{Nef}

\DeclareMathOperator{\spec}{Spec}

\DeclareMathOperator{\GL}{GL}

\newcommand{\DivInf}{\Div_\infty}
\DeclareMathOperator{\car}{char}

\DeclareMathOperator{\BD}{\partial_X X_0}
\DeclareMathOperator{\trdeg}{tr.deg}
\DeclareMathOperator{\reldeg}{reldeg}
\DeclareMathOperator{\codim}{codim}
\DeclareMathOperator{\Weil}{Weil}
\DeclareMathOperator{\Cartier}{Cartier}
\DeclareMathOperator{\Cinf}{\Cartier_\infty (X_0)}
\DeclareMathOperator{\Winf}{\Weil_\infty (X_0)}

\DeclareMathOperator{\cNk}{c-N^k}
\DeclareMathOperator{\wNk}{w-N^k}
\DeclareMathOperator{\cNone}{c-N^1}
\DeclareMathOperator{\wNone}{w-N^1}
\DeclareMathOperator{\cNN}{c-N}
\DeclareMathOperator{\wNN}{w-N}
\DeclareMathOperator{\cBPF}{c-BPF}
\DeclareMathOperator{\wBPF}{w-BPF}

\DeclareMathOperator{\Vinf}{\cV_\infty}
\DeclareMathOperator{\BPF}{BPF}
\DeclareMathOperator{\Vect}{Vect}

\renewcommand{\aa}{\mathfrak a}
\newcommand{\bb}{\mathfrak b}

\newtheorem{thm}{Theorem}[section]
\newtheorem{thm*}{Theorem}
\newtheorem{bigthm}{Theorem}
\newtheorem{prop}[thm]{Proposition}
\newtheorem{cor}[thm]{Corollary}
\newtheorem{lemme}[thm]{Lemma}

\newtheorem{claim}[thm]{Claim}
\theoremstyle{definition}
\newtheorem{dfn}[thm]{Definition}
\newtheorem{ex}[thm]{Example}
\newtheorem{rmq}[thm]{Remark}

\AtBeginEnvironment{dfn}{%
  \setlist[enumerate]{label={(\roman*)}}
}

\AtBeginEnvironment{thm}{%
  \setlist[enumerate,1]{label={(\arabic*)}}
}

\AtBeginEnvironment{prop}{%
  \setlist[enumerate,1]{label={(\arabic*)}}
}

\AtBeginEnvironment{lemme}{%
  \setlist[enumerate,1]{label={(\arabic*)}}
}

\begin{document}
\title{Dynamical degrees of twisted rational maps}
\author{Marc Abboud}
\address{Marc Abboud, Universit\'e de Neuch\^atel,
Rue Emile-Argand 11,
CH-2000 Neuch\^atel,
Switzerland}
\email{marc.abboud@normalesup.org}

\author{Junyi Xie}
\address{Junyi Xie, Beijing International Center for Mathematical Research, Peking University, Beijing 100871, China}
\email{xiejunyi@bicmr.pku.edu.cn}
\begin{abstract}
	Twisted rational maps arise naturally in relative algebraic dynamics: if a rational self-map preserves a fibration, then the induced map on the generic fiber is usually not an ordinary rational self-map over the function field, but a twisted one.
		This suggests that twisted rational maps form a natural framework for studying relative dynamics.
		In this framework we extend the theory of dynamical degrees, the numerical invariants measuring the asymptotic complexity of a dynamical system: the defining limits exist, are independent of the choice of polarization, and are birational invariants.
		We also identify the relative dynamical degrees of a semi-conjugacy with the dynamical degrees of the induced twisted rational map on the generic fiber, and prove the corresponding mixed degree formula.
		Finally, using the spectral interpretation of dynamical degrees and valuative methods at infinity, we prove an algebraicity result for the first dynamical degree of twisted endomorphisms of affine varieties.
	
\end{abstract}
\maketitle

\tableofcontents
\section{Introduction}\label{sec:intro}
Let $K$ be any field and let $f:X \dashrightarrow X$ be a rational transformation of a projective variety over $K$. The
dynamical degrees of $f$ are defined as
\begin{equation}
	\lambda_k (f) := \lim_{n \to \infty} \left( (f^n)^* H^k \cdot H^{d-k} \right)^{1/n}
\end{equation}
where $d = \dim X$, see for example \cite{dangDegreesIteratesRational2020, truongRelativeDynamicalDegrees2015}.
It is a fundamental dynamical invariant of rational transformations.
We denote the quantity $(f^n)^* H^k \cdot H^{d-k}$ by $\deg_{k, H} (f^n)$.
In this paper we propose twisted rational maps as a natural framework for extending the theory of dynamical degrees to relative algebraic dynamics.
We show that the classical construction for rational maps extends to twisted rational maps.
Let $X,Y$
be varieties over a field $K$, a \emph{twisted morphism} from $X$ to $Y$ is a scheme morphism of finite type $f_\tau : X
	\rightarrow Y$ such that there exists a finite map $\tau : \spec K \rightarrow \spec K$ making the following diagram
\begin{equation}
	\begin{tikzcd}
		X \ar[r, "f_\tau"] \ar[d] & Y \ar[d] \\ \spec K \ar[r, "\tau"] & \spec K
	\end{tikzcd}
\end{equation}
commutative.
We define in a similar way \emph{twisted rational maps}.
The motivating example of such maps comes from dynamics.
Suppose one has a rational transformation $f :X \dashrightarrow X$ preserving a fibration $\pi :X \dashrightarrow Y$, that is there exists $g : Y \dashrightarrow Y$ such that $\pi \circ f = g \circ \pi$.
If $\eta \in Y$ is the generic point of $Y$, then the generic fiber $X_\eta := \pi^{-1} (\eta)$ is a $K(Y)$-variety and the induced map $f_\eta: X_\eta \dashrightarrow X_\eta$ is a twisted rational map.

\subsection{Dynamical degrees}\label{subsec:dyn-degres-intro}
We show that the construction of dynamical degrees for dominant rational maps carries through for twisted rational maps.
The main new technical point is that intersection numbers depend on the chosen structural morphism to $\spec K$, while Picard groups, cycles, and numerical classes are canonically identified after twisting.
\begin{bigthm}
	\label{thm:dyn-degrees}
	Let $X$ be a projective variety over a field $K$ and let $f_\tau : X \dashrightarrow X$ be a dominant twisted rational
	map, then, for every $i = 1, \dots, \dim X$ and for any big and nef divisor $H$ over $X$, the limit
	\begin{equation}
		\lambda_i (f_\tau) = \lim_n \left( (f_\tau^n)^* H^i \cdot H^{d - i} \right)^{1/n}
	\end{equation}
	exists.
	It does not depend on the choice of the big and nef divisor $H$ and it is a birational invariant.
\end{bigthm}
The main example of twisted rational maps is the following.
Let $X,Y$ be projective varieties with two $K$-rational maps $f : X \dashrightarrow X$ and $g : Y \dashrightarrow Y$ which are semiconjugated by a fibration $q : X \rightarrow Y$.
Let $\eta$ be the generic point of $Y$, the generic fiber $X_\eta = X \times_Y \spec K(\eta)$ is a variety over $K(Y)$ the function field of $Y$.
Then, the restriction of $f$ to the generic fiber induces a twisted rational map $f_\eta : X_\eta \dashrightarrow X_\eta$ where twisting on the base is given by $g$.
In this setting, there is a \emph{relative dynamical degree} introduced in \cite{truongRelativeDynamicalDegrees2015} which is denoted by $\lambda_k (f_{|q})$.
It measures the asymptotic growth of the degrees of the iterates of $f$ in the directions tangent to the fibers of $q$: roughly speaking, if $l=\dim Y$, then $\lambda_k(f_{|q})$ is the exponential growth rate of intersection numbers of the form
\[
	(f^n)^*H_X^k \cdot H_X^{\dim X-l-k}\cdot q^*H_Y^l,
\]
where $H_X$ and $H_Y$ are big and nef divisors on $X$ and $Y$ respectively.
A precise definition is recalled in \S\ref{subsec:relative-dynamical-degrees}.

These relative dynamical degrees were introduced to prove a mixed degree formula for semi-conjugated rational maps.
We show that these relative dynamical degrees are exactly the dynamical degrees of the twisted rational map induced on the generic fiber of the fibration and we prove that the mixed degree formula also holds for dynamical degrees of semi-conjugated twisted rational maps.

\begin{bigthm}
	\label{bigthm:relative-dyn-degrees-formula}
	If $f : X \dashrightarrow X$ and $g : Y \dashrightarrow Y$ are dominant rational maps of $K$-varieties semi-conjugated by a dominant morphism $q : X \rightarrow Y$, then
	\begin{equation}
		\lambda_k (f_{|q}) = \lambda_k (f_\eta).
	\end{equation}
	More generally, let $f_\tau : X \dashrightarrow X$ and $g_\omega : Y \dashrightarrow Y$ be dominant twisted rational maps and let $q_\sigma : X \dashrightarrow Y$ be a dominant twisted rational map such that
	\[
		q_\sigma \circ f_\tau = g_\omega \circ q_\sigma.
	\]
	Then we have the mixed degrees formula
	\begin{equation}
		\lambda_k (f_\tau) = \max_{\substack{s+t = k\\ s\leq \dim X-\dim Y,\ t\leq \dim Y}} \left( \lambda_s (f_\eta) \lambda_t (g_\omega) \right),
	\end{equation}
	where $f_\eta$ is the twisted rational map induced by $f_\tau$ on the generic fiber of $q_\sigma$.
\end{bigthm}
\begin{rmq}
	It is sometimes convenient to reformulate the mixed degree formula in terms of dynamical Lyapunov multipliers.
	For a dominant (twisted) rational self-map $h$ of a $d$-dimensional variety, define
	\[
		\mu_i(h):=\frac{\lambda_i(h)}{\lambda_{i-1}(h)},\qquad i=1,\dots,d,
	\]
	with the convention $\lambda_0(h)=1$.
	These quantities were introduced in \cite{xieAlgebraicDynamicsRecursive2025}; here we call them dynamical Lyapunov multipliers.
	The log-concavity of dynamical degrees implies that the sequence $(\mu_i(h))$ is non-increasing.
	If $e=\dim X-\dim Y$ and $\ell=\dim Y$, then the mixed degree formula is equivalent to
	\[
		\bigl(\mu_1(f_\tau),\dots,\mu_{\dim X}(f_\tau)\bigr)
		=
		\operatorname{sort}_{\geq}\bigl(
		\mu_1(f_\eta),\dots,\mu_e(f_\eta),
		\mu_1(g_\tau),\dots,\mu_\ell(g_\tau)
		\bigr),
	\]
	where $\operatorname{sort}_{\geq}$ means that the listed multipliers are rearranged in non-increasing order.
\end{rmq}

The proof of the mixed degrees formula differs from \cite{truongRelativeDynamicalDegrees2015} and gives a more elementary argument even for rational transformations of varieties.

The construction of the dynamical degrees follows from the following principle: if $X$ is a scheme, the notion of invertible sheaves and dimension of closed subsets is independent of its structure as a $K$-variety.
Since dynamical degrees are computed using intersection of line bundles, the theory developed for rational transformations of varieties goes through for twisted rational maps.
The only difference to account for is that the definition of the degrees of zero cycles actually depends on the structure of $K$-varieties by multiplying intersection numbers by the degree of $\tau$.
Taking this into account suffices to define properly the theory of dynamical degrees for twisted
rational maps.

In \cite{dangSpectralInterpretationsDynamical2021}, Dang and Favre showed that the dynamical degrees of $f : X \dashrightarrow X$ can be interpreted as spectral radii of the continuous operator induced by $f$ on Banach spaces associated to the space of numerical classes of every birational model of $X$.
We show that this construction carries through in our setting and that the arguments adapt without substantial changes.
In particular, Theorem 1 of
\cite{dangSpectralInterpretationsDynamical2021} also holds in our setting
\begin{thm}
	\label{thm:croissance-degrees}
	Let $f_\tau: X \dashrightarrow X$ be a dominant twisted rational map such that $\lambda_1 (f_\tau)^2 > \lambda_2
		(f_\tau)$, then for any big and nef class $H$ over $X$ there exists a constant $C > 0$ such that
	\begin{equation}
		\deg_{1, H} (f_\tau^n) = C \lambda_1 (f_\tau)^n + O(\lambda^n)
	\end{equation}
	for any $\sqrt{\lambda_2 (f_\tau)} < \lambda < \lambda_1 (f_\tau)$.
\end{thm}

\subsection{Dynamical degrees over affine varieties}\label{subsec:dyn-degrees-affine-varieties}
In \cite{dangSpectralInterpretationsDynamical2021}, Dang and Favre showed that for any proper polynomial endomorphism $f : \A^d_K \rightarrow \A^d_K$ such that $\lambda_1 (f)^2 > \lambda_2 (f)$, where $K$ is a field of characteristic zero and $\A^d_K$ is the affine space over $K$, the first dynamical degree $\lambda_1 (f)$ is an algebraic number of degree $\leq d$.
This is shown using their spectral interpretation of dynamical degrees and in particular Theorem \ref{thm:croissance-degrees}.
We generalise their result to twisted endomorphisms of affine varieties, also dropping the properness assumption.

\begin{bigthm}
	\label{bigthm:dyn-degrees-affine-varieties}
	Let $X_0$ be a smooth affine variety of dimension $d$ over a field of characteristic zero and let $f_\tau : X_0 \rightarrow X_0$ be a dominant twisted endomorphism.
	If $\lambda_1 (f_\tau)^2 > \lambda_2 (f_\tau)$, then $\lambda_1 (f_\tau)$ is an algebraic number of degree $\leq (d-1)^2$ if $d\geq 3$, and $\leq d$ if $d=1,2$.

  Furthermore, if the quasi-Albanese variety of $X_0$ is trivial and $\lambda_1 (f_\tau)^2 > \lambda_2 (f_\tau)$ then $\lambda_1 (f_\tau)$ is an algebraic number of degree $\leq d$.
\end{bigthm}
Since the quasi-Albanese variety of the affine space $\A^d_K$ is trivial we recover the generalisation of Theorem 2 of \cite{dangSpectralInterpretationsDynamical2021}.

The strategy of proof is as follows.
Let $\QAlb(X_0)$ be the quasi-Albanese variety of $X_0$.
If $\QAlb(X_0)$ is trivial then we deduce the result using valuative techniques as well following \cite{dangSpectralInterpretationsDynamical2021} and the work of the first author in \cite{abboudDynamicsEndomorphismsAffine2023} on affine surfaces.
Otherwise, the twisted endomorphism $f_\tau$ induces a twisted endomorphism $g_\tau : \QAlb(X_0) \rightarrow \QAlb(X_0)$ which is semiconjugated to $f_\tau$ by the quasi-Albanese map $q : X_0 \rightarrow \QAlb(X_0)$, the result will then follow by induction using Theorem \ref{bigthm:relative-dyn-degrees-formula}.

\subsection*{Organization of the paper}
The paper is organized as follows.
	In \S\ref{sec:definitions} we introduce twisted rational maps and the basic intersection-theoretic operations needed to define their degrees.
	In \S\ref{subsec:relative-dynamical-degrees} and \S\ref{subsec:siu-inequality} we prove the comparison with relative dynamical degrees and the mixed degree formula.
	In the following sections we adapt the Banach-space interpretation of dynamical degrees and develop the valuative tools at infinity.
	The final section applies these results to prove the algebraicity theorem for twisted endomorphisms of affine varieties.

\section{Definitions}\label{sec:definitions}
Let $K$ be a field and $X, X'$ be varieties over $K$, i.e. integral schemes of finite type over $K$.
A twisted rational map of $K$-varieties from $X$ to $X'$ is a rational map $f_\tau : X \dashrightarrow X'$
such that there exists a finite map $\tau : \spec K
	\rightarrow \spec K$ (equivalently, a finite field homomorphism $\tau^*:K\to K$) making the following diagram
\begin{equation}
	\begin{tikzcd}
		X \ar[r, dashed, "f_\tau"] \ar[d] & X' \ar[d] \\ \spec K \ar[r, "\tau"] & \spec K
	\end{tikzcd}
\end{equation}
commute.
A twisted morphism is a twisted rational map where the map $f_\tau : X \dashrightarrow X'$ is regular.
The composition of twisted rational maps and twisted morphisms is well defined.
\begin{rmq}
	If $K$ is algebraically closed, then necessarily $\tau$ is an automorphism.
\end{rmq}

Let $\tau : \spec K \rightarrow \spec K$ be a finite map and $q : X \rightarrow \spec K$ a $K$-variety, we write $X^{(\tau)}$ for the $K$-variety with underlying scheme $X$ and structural morphism $\tau \circ q$.
If $f_\tau : X
	\dashrightarrow Y$ is a twisted rational map then it defines a rational map of $K$-varieties
\begin{equation}
	f : X^{(\tau)} \dashrightarrow Y.
\end{equation}
And conversely any such rational map defines a twisted rational map.
\subsection{Main example}\label{subsec:main-example}
Let $q : X \rightarrow Y$ be a dominant morphism of irreducible $K$-varieties and
let $f: X \dashrightarrow X$, $g: Y \dashrightarrow Y$ be rational maps of $K$-varieties such that the following diagram
\begin{equation}
	\begin{tikzcd}
		X \ar[r, dashed, "f"] \ar[d, "q"] & X \ar[d,"q"]\\ Y \ar[r, dashed, "g"] & Y
	\end{tikzcd}
	\label{eq:diagram-fibration}
\end{equation}
commutes.
Let $\eta \in Y$ be the generic point of $Y$ and let $K(Y)$ be the function field of $Y$.
Taking the fiber
	over $\eta$ in \eqref{eq:diagram-fibration}, we get a commutative diagram
\begin{equation}
	\begin{tikzcd}
		X_\eta \ar[r,dashed,"f_\eta"] \ar[d, "q_\eta"] & X_\eta \ar[d, "q_\eta"]\\ \spec K(Y) \ar[r, "g_\eta"] & \spec K(Y)
	\end{tikzcd}
	.
\end{equation}
Since $g: Y \dashrightarrow Y$ is a dominant rational self-map of the integral variety $Y$, it induces a finite extension of function fields.
	Equivalently, $g_\eta : \spec K(Y) \rightarrow \spec K(Y)$ is finite.
Thus $f_\eta : X_\eta \dashrightarrow X_\eta$ is a twisted rational map of $K(Y)$-varieties.

\begin{lemme}
	\label{lemme:change-of-base}
	Let $\tau : \spec K \rightarrow \spec K$ be a finite map and $X$ be a $K$-variety.
	Suppose there is a morphism of $K$-varieties $f: Y \rightarrow X^{(\tau)}$, then there exists $q' : Y \rightarrow \spec K$ such that the structural morphism $q_Y : Y \rightarrow \spec K$ satisfies $q_Y = \tau \circ q' $.
	If $Z$ is the $K$-variety with underlying scheme $Y$ and structural morphism $q'$ then $Y = Z^{(\tau)}$ and $Z \rightarrow X$ is a morphism of $K$-varieties.
\end{lemme}
\begin{proof}
	Let $q_X:X\to\spec K$ and $q_Y:Y\to\spec K$ be the structural morphisms.
	Since $f:Y\to X^{(\tau)}$ is a morphism of $K$-varieties, its structural morphisms satisfy \[ q_Y=(\tau\circ q_X)\circ f .
	\]
	Set
	\[
		q' := q_X\circ f : Y\to \spec K .
	\]
	Then $q_Y=\tau\circ q'$.
	Let $Z$ be the scheme $Y$ endowed with the structural morphism $q'$.
	By definition, the twist $Z^{(\tau)}$ has underlying scheme $Y$ and structural morphism $\tau\circ q'=q_Y$, hence $Z^{(\tau)}=Y$ as $K$-varieties.
	Finally, the same underlying morphism $f:Y\to X$ gives a morphism $Z\to X$ of $K$-varieties, because its compatibility with the structural morphisms is precisely the identity $q'=q_X\circ f$.
\end{proof}

\begin{rmq}
	\label{rmq:not-base-change}
	In general we do not have that $X = X^{(\tau)} \times_K \spec K$.
	Indeed, this is not even true if $X = \spec K$.
	Here the first map is the structural morphism of $X^{(\tau)}$, and the right-hand copy of $\spec K$ maps to the base via $\tau$.
	
	example take $K = \Q (t)$ and let $\tau$ be the $\Q$-morphism that sends $t$ to $t^2$. Then, $\Q(t) = \Q(t^2) [x] /
		(x^2 - t^2)$ and thus
	\begin{equation}
		\Q(t) \otimes_{\Q(t^2)} \Q(t) = \Q(t) [x] / (x -t)(x+t)
	\end{equation}
	which is not isomorphic to $\Q(t)$ as it is not even a domain.
\end{rmq}

\subsection{Affine and projective space}\label{subsec:affine-and-projective-space}
Let $X = \A^n_K$.
Twisted morphisms to $X$ have a particularly simple form. If $\tau : \spec K \rightarrow \spec K$ is a
finite map, we write $\hat \tau : X \rightarrow X$ for the twisted morphism whose pullback on the coordinate ring is defined by
\begin{equation}
	\hat \tau^* \left(\sum_I a_I x^I\right) = \sum_I \tau^*(a_I) x^I.
\end{equation}
The same semilinear rule on homogeneous coordinates defines a twisted morphism $\hat \tau : \P^n_K \rightarrow \P^n_K$.

\begin{prop}
	\label{prop:twisted-morphism-affine-space}
	Write $X = \A^n_K$ or $X = \P^n_K$, let $Y$ be a $K$-variety.
	Every twisted morphism $\phi_\tau : Y \rightarrow X$ is of the form
	\begin{equation}
		\phi_\tau = \hat \tau \circ \phi
	\end{equation}
	where $\phi : Y \rightarrow X$ is a morphism of $K$-varieties.

	Furthermore, if $Y$ is a projective variety and $\phi_\tau : Y \dashrightarrow \P^n_K$ is a twisted rational map, then there exists a rational map of $K$-varieties $\phi : Y \dashrightarrow \P^n_K$ such that $\phi_\tau = \hat \tau \circ \phi$.
\end{prop}
\begin{proof}
	We first do the proof when $X = \A^n_K$.
	A scheme morphism $f: Y \rightarrow \A^n_K$ is equivalent to a ring homomorphism $f^* : K [x_1, \dots, x_n] \rightarrow \OO(Y)$.
	We compute $\phi_\tau^* (P)$ for any polynomial $P$.
	We get
	\begin{equation}
		\phi_\tau^* \left( \sum_I a_I x_1^{i_1} \cdots x_n^{i_n} \right) = \sum_I \tau^* (a_I) \phi_\tau^*
		(x_1)^{i_1} \cdots \phi_\tau^* (x_n)^{i_n}.
	\end{equation}
	Write $y_i = \phi_\tau^* (x_i) \in \OO (Y)$ and let $\phi : Y \rightarrow \A^n_K$ be the morphism of $K$-varieties defined by $\phi^* (x_i) = y_i$.
	We have that $\phi_\tau = \hat \tau \circ \phi$.

	For $X = \P^n_K$, the proof is similar.
	Let $L$ be the line bundle over $Y$ defined by $L = \phi_\tau^* \OO (1)$.
	Let $s_i \in H^0 (Y, L)$ be the global sections such that $\phi_\tau^* X_i = s_i$.
	Then, there is a unique morphism of $K$-varieties $\phi : Y \rightarrow \P^n_K$ such that $\phi^* \OO (1) = L$ and $s_i = \phi^* X_i$ and we have that $\phi_\tau = \hat \tau \circ \phi$.

	The last case follows because if $\phi_\tau : Y \dashrightarrow \P^N_K$ is a twisted rational map, then there exists a blowup $Z \rightarrow Y$ such that the lift of $\phi_\tau$ is a twisted morphism $\phi_\tau : Z \rightarrow \P^N_K$.
\end{proof}

\begin{cor}
	\label{cor:twisted-endo-alg-torus}
	Every twisted endomorphism $\phi_\tau : \G_m^K \rightarrow \G_m^K$ is of the form $\phi_\tau = \hat \tau \circ \phi$ where $\phi$ is an endomorphism of $\G_m^K$.
\end{cor}

\subsection{Toric varieties}\label{subsec:toric-varieties}
Let $N\simeq \Z^d$ be a lattice, let $M=\operatorname{Hom}(N,\Z)$, and let $\Sigma$ be a fan in $N_\R$.
For a cone $\sigma\in\Sigma$, write
\[
	U_\sigma=\spec K[\sigma^\vee\cap M].
\]
If $\tau:\spec K\to\spec K$ is finite, then $U_\sigma^{(\tau)}$ is the same scheme endowed with the structural morphism twisted by $\tau$.
Equivalently, on coordinate rings the canonical twisted morphism
\[
	\hat\tau_\sigma:U_\sigma\longrightarrow U_\sigma
\]
is defined by the semilinear rule
\[
	\hat\tau_\sigma^*\left(\sum_m a_m\chi^m\right)=\sum_m \tau^*(a_m)\chi^m .
\]
This rule leaves the monomials $\chi^m$ unchanged and only applies $\tau^*$ to the coefficients.
Since the inclusions of affine toric charts are induced by localizing at monomials, these semilinear maps are compatible on overlaps and glue to a twisted morphism
\[
	\hat\tau:X_\Sigma\longrightarrow X_\Sigma .
\]
In particular, the twist does not change the lattice, the fan, the rays, or the corresponding torus-invariant prime divisors; only the $K$-structure is changed.

\begin{prop}
	\label{prop:toric-varieties}
	Let $K$ be a field, let $\tau:\spec K\to\spec K$ be a finite map, and let $X_\Sigma$ be a split toric variety over $K$.
	Then the twisted endomorphism $\hat\tau:\G_m^d\to\G_m^d$ extends canonically to a twisted endomorphism
	\[
		\hat\tau:X_\Sigma\longrightarrow X_\Sigma .
	\]
	Moreover $X_\Sigma$ and $X_\Sigma^{(\tau)}$ are canonically described by the same fan $\Sigma$, and the induced identification sends each torus-invariant prime divisor corresponding to a ray $\rho\in\Sigma(1)$ to the torus-invariant prime divisor corresponding to the same ray.
\end{prop}

\begin{proof}
	For every cone $\sigma\in\Sigma$, the formula
	\[
		\sum_m a_m\chi^m\longmapsto \sum_m \tau^*(a_m)\chi^m
	\]
	defines a homomorphism
	$K[\sigma^\vee\cap M]\to K[\sigma^\vee\cap M]$ which is semilinear with respect to $\tau^*:K\to K$.
	Hence it defines a twisted morphism $\hat\tau_\sigma:U_\sigma\to U_\sigma$.
	If $\sigma'\preceq\sigma$, then the open immersion $U_{\sigma'}\subset U_\sigma$ is obtained by inverting the monomials corresponding to $\sigma^\vee\cap M$ which are units on $U_{\sigma'}$.
	The above formula fixes every monomial $\chi^m$, so it commutes with these localizations.
	Thus the maps $\hat\tau_\sigma$ glue to a twisted morphism $\hat\tau:X_\Sigma\to X_\Sigma$.
	On the dense torus this is precisely the coefficientwise twisted endomorphism of $\G_m^d$.
	The same affine charts, cones, and face relations describe $X_\Sigma^{(\tau)}$; therefore the fan and its rays are canonically identified before and after twisting, and the corresponding torus-invariant prime divisors are identified as claimed.
\end{proof}

\subsection{Intersection theory and BPF classes}\label{subsec:intersection-theory}\label{subsec:numerical-classes}\label{subsec:ample-divisors-bpf-classes}
We first recall the intersection-theoretic conventions used below, independently of twisted morphisms.
Following \cite{fultonIntersectionTheory1998}, for a normal projective $K$-variety $X$, we write $Z_m(X)$ for the group of formal combinations of irreducible subvarieties of dimension $m$, and $A_m(X)$ for the quotient modulo rational equivalence.
If $X$ is equidimensional of dimension $d$, we also use the codimension notation $Z^k(X):=Z_{d-k}(X)$ and $A^k(X):=A_{d-k}(X)$.

If $W\subset X$ is an irreducible subvariety of dimension $m+1$ and $r\in K(W)^\times$, the principal divisor $\div(r)$ is an $m$-cycle on $X$.
These principal divisors generate rational equivalence, and the Chow group $A_m(X)$ is the corresponding quotient of $Z_m(X)$.
A Cartier divisor $D$ defines a first Chern class
\begin{equation}
	c_1 : \Pic(X) \rightarrow A^1 (X)
\end{equation}
and an intersection homomorphism
\begin{equation}
	c_1 (D) \cdot : A^k (X) \rightarrow A^{k+1} (X)
\end{equation}
in codimension notation. Finally, the degree map on $0$-cycles is
\begin{equation}
	\deg \left( \sum_i a_i [p_i] \right) = \sum_i a_i [K(p_i) : K]
\end{equation}
for closed points $p_i$. If $D_1,\ldots,D_k$ are Cartier divisors and $[\alpha]\in A_k(X)$, then
\begin{equation}
	D_1\cdots D_k\cdot[\alpha]:=\deg(c_1(D_1)\cdots c_1(D_k)\cdot[\alpha]).
\end{equation}

\medskip
\noindent\emph{Numerical cycle classes.}
We will also use flat pullbacks of cycles.
If $q:X\rightarrow Y$ is flat of relative dimension $e$, then
\begin{equation}
	q^*:A_m(Y)\rightarrow A_{m+e}(X),\qquad q^*[V]=[q^{-1}(V)].
	\label{eq:flat-pullback-cycles}
\end{equation}

Two $m$-cycles $\alpha,\beta$ on a normal projective variety $X$ are numerically equivalent, in the sense of
\cite{dangDegreesIteratesRational2020}, if for every flat morphism $q:X_1\rightarrow X$ of relative dimension $e$ and
all Cartier divisors $D_1,\ldots,D_{e+m}$ on $X_1$ we have
\begin{equation}
	D_1\cdots D_{e+m}\cdot q^*\alpha=D_1\cdots D_{e+m}\cdot q^*\beta.
	\label{eq:numerical-equivalence-dimension}
\end{equation}
We write $N_m(X)$ for the group of numerical classes of $m$-cycles and set \[ N_m(X)_{\mathbb R}:=N_m(X)\otimes_{\mathbb Z}\mathbb R.
\]
Unless otherwise specified, all cones of cycle classes below are cones in these real vector spaces.

\medskip
\noindent\emph{Dual numerical classes.}
We also use the dual numerical groups of \cite{dangDegreesIteratesRational2020}.
The group $N^k(X)$ consists of dual numerical classes of codimension $k$, and we put $N^k(X)_{\mathbb R}:=N^k(X)\otimes_{\mathbb Z}\mathbb R$.
The
direct sum $N^\bullet(X)_{\mathbb R}$ is a graded ring, and it acts on dimension classes by
\begin{equation}
	N^k(X)_{\mathbb R}\times N_m(X)_{\mathbb R}\longrightarrow N_{m-k}(X)_{\mathbb R},\qquad (\alpha,z)\longmapsto \alpha\cdot z.
	\label{eq:dual-numerical-action}
\end{equation}
Thus, on a normal projective variety, products below are products in the dual numerical ring and its induced action on $N_\bullet(X)$; they are not arbitrary intersections of ordinary cycles.
When $X$ is smooth of dimension $d$, the canonical map $N^k(X)_{\mathbb Q}\rightarrow N_{d-k}(X)_{\mathbb Q}$ is an isomorphism, recovering the usual notation for numerical cycle classes.

\medskip
\noindent\emph{Divisorial positivity.}
Recall first that a line bundle $L$ on a projective $K$-variety $X$ is \emph{very ample} if it induces a closed
immersion
\begin{equation}
	\phi:X\hookrightarrow \P^N_K,\qquad L\simeq \phi^*\OO(1).
	\label{eq:very-ample-line-bundle}
\end{equation}
It is \emph{ample} if $L^{\otimes m}$ is very ample for some $m>0$.
A Cartier divisor $D$ is very ample, respectively ample, if the line bundle $\OO_X(D)$ is very ample, respectively ample.
The ample cone $\operatorname{Amp}^1(X)\subset N^1(X)_{\mathbb R}$ is the cone generated by classes of ample Cartier divisors; its elements are called ample real divisor classes.

For real divisor classes in $N^1(X)_{\mathbb R}$, we use the usual cones $\Nef^1(X)$, $\operatorname{Psef}^1(X)$, and $\operatorname{Big}^1(X)$.
A real divisor class is \emph{nef} if it has non-negative intersection with every integral curve on $X$.
It is \emph{pseudoeffective} if it lies in the closure of the cone generated by effective divisor classes.
We say that it is \emph{big} if it belongs to the interior of the pseudoeffective cone; equivalently, by the Kodaira lemma, it can be written as the sum of an ample real class and an effective real class.
This is the form of bigness that we shall use below.

\medskip
\noindent\emph{BPF classes.}
We recall the definition of basepoint free classes in the normal projective setting from \cite[\S 3.3]{dangDegreesIteratesRational2020}.
Let $X$ be a normal projective variety.
A class in
$N^i(X)_{\mathbb R}$ is \emph{basepoint free} if it belongs to the closure of the convex cone generated by products
\begin{equation}[\gamma_1]\cdots[\gamma_r],
	\label{eq:normal-bpf-generators}
\end{equation}
where $i_1+\cdots+i_r=i$ and each factor $[\gamma_j]\in N^{i_j}(X)_{\mathbb R}$ is induced by a complete intersection
\begin{equation}
	\gamma_j=A_{j,1}\cdots A_{j,e_j+i_j}
	\label{eq:bpf-complete-intersection-generator}
\end{equation}
of ample Cartier divisors on an equidimensional projective scheme $p_j:X_j\rightarrow X$ that is flat of relative
dimension $e_j$.
We denote this cone by $\BPF^i(X)\subset N^i(X)_{\mathbb R}$.

By \cite[Theorem 3.3.3]{dangDegreesIteratesRational2020}, products of BPF classes are BPF, and finite products of BPF classes can be evaluated on dimension-cycle classes through the action \eqref{eq:dual-numerical-action}.
In the smooth case this agrees with the Fulger--Lehmann strongly basepoint free cone.
Moreover, in codimension one, Dang proves that \[ \BPF^1(X)=\Nef^1(X), \] see \cite[Theorem 3.3.3(vi)]{dangDegreesIteratesRational2020}.
Thus BPF classes are a higher-codimensional analogue of nef divisor classes.

\subsection{Twisted intersection-theoretic operations}\label{subsec:twisted-intersection-operations}
We now record how the preceding constructions behave after changing the structural morphism and under twisted morphisms.

\begin{lemme}
	\label{lemme:changing-structure-morphism}
	Let $X$ be a scheme with two morphisms $\tau_i : X \rightarrow \spec K, i=1,2$ that define two $K$-variety structure on $X$.
	We write $X_1, X_2$ for the two $K$-varieties, we have canonical isomorphisms
	\begin{equation}
		Z_m(X_1) \simeq Z_m (X_2) \text{ and } A_m(X_1) \simeq A_m (X_2).
	\end{equation}
	We also have a canonical isomorphism of the Picard groups
	\begin{equation}
		\Pic (X_1) \simeq \Pic(X_2)
	\end{equation}
	which is compatible with the first Chern class
	\begin{equation}
		c_1 (\Pic (X_i)) \rightarrow A^1 (X_i).
	\end{equation}
	Furthermore, if $D$ is a Cartier divisor over $X$, then the induced map
	\begin{equation}
		c_1 (D) \cdot : A^k (X_i) \rightarrow A^{k+1} (X_i)
	\end{equation}
	is compatible with this isomorphism.
	In particular, if $X$ is a $K$-variety and $\tau:\spec K\rightarrow \spec K$ is finite, then the two $K$-structures on the underlying scheme of $X$ and $X^{(\tau)}$ give canonical isomorphisms $A_m(X)\simeq A_m(X^{(\tau)})$ and $\Pic(X)\simeq \Pic(X^{(\tau)})$.

\end{lemme}
\begin{proof}
	Let $Y$ be an irreducible closed subset of $X$ and let $\eta_Y$ be its generic point.
	Then viewing $Y$ as a subvariety of $X_1$ and $X_2$, we have two $K$-structure $Y_1,Y_2$ over $Y$.
	The dimension and codimension of $Y_i$ in $X_i$ do not depend on $i$ since they are intrinsic to the underlying topological space.
	More precisely, $\dim Y$ is the Krull dimension of the irreducible closed subset $\overline{\{\eta_Y\}}$, and $\codim_XY$ is its codimension in the underlying topological space of $X$.

	We thus have an isomorphism
	\begin{equation}
		\sum_Y a_Y Y_1 \in Z_m (X_1) \mapsto \sum_Y a_Y Y_2 \in Z_m (X_2).
	\end{equation}
	We have to show that this isomorphism carries through when we take classes of cycles.
	Let $W$ be a closed irreducible subset of $X$ of dimension $m+1$, if $V \subset W$ is a closed subvariety of dimension $m$, then again the function $\ord_V$ does not depend on the structural morphism to $K$ nor does the field of rational functions of $W$.
	If $r$ is a
	rational function over $W$, then
	\begin{equation}
		\div_{W_1} (r) = \sum_{V \subset W} \ord_V (r) V_1 \mapsto \sum_{V \subset W} \ord_V (r) V_2
	\end{equation}
	so \emph{rationally equivalent to zero} classes in $X_1$ map to rationally equivalent to zero classes in $X_2$ and we
	thus have the isomorphism
	\begin{equation}
		A_m (X_1) \simeq A_m (X_2).
	\end{equation}
	Similarly, the notion of Cartier divisor is independent of the structural morphism so that $\Pic (X_1)\simeq \Pic(X_2)$ canonically and the compatibility with the first Chern class and the intersection homomorphism is also proven.
\end{proof}
However, intersection numbers may depend on the structural morphism.

\begin{lemme}
	\label{lemme:twist-degree-map}
	Let $X$ be a $K$ variety and $\tau : \spec K \rightarrow \spec K$ a finite map.
	The degree map $\deg : A_0 (X) \rightarrow \R$ and $\deg^\tau : A_0 (X^{(\tau)}) \rightarrow \R$ are different.
	We have the following relation
	\begin{equation}
		\deg^\tau = \deg (\tau) \deg.
	\end{equation}
\end{lemme}
\begin{proof}
	The lemma follows directly from the fact that $[K : \tau(K)] = \deg (\tau)$.
\end{proof}
When defining dynamical degrees this will have to be taken into account, however we will see in \S \ref{subsec:numerical-classes} that the groups of numerical classes of $X$ and $X^{(\tau)}$ are canonically isomorphic.

\begin{lemme}
	\label{lemme:canonical-divisor-unchanged}
	Let $X$ be a smooth variety over a field $K$ and let $\tau : \spec K \rightarrow \spec K$ be a finite separable map, then after identifying the underlying schemes of $X$ and $X^{(\tau)}$, the sheaf of differential forms $\Omega_{X/K}$ and $\Omega_{X^{(\tau)} / K}$ are canonically isomorphic as sheaves of $\OO_X$-modules.
	In particular, under the canonical isomorphism $\Pic (X) \simeq \Pic (X^{(\tau)})$, we have equality of the canonical
	sheaves
	\begin{equation}
		\omega_X = \omega_{X^{(\tau)}}.
	\end{equation}
\end{lemme}
\begin{proof}
	For a scheme morphism $X \rightarrow Y$, we write $\Omega_{X / Y}$ for the relative sheaf of differentials.
	In particular, $\omega_X = \bigwedge^{\dim X} \Omega_{X/K}$.
	Consider the composition of scheme morphism $X \xrightarrow{q} \spec K \xrightarrow{\tau} \spec K^\tau$.
	Then, by
	\cite{hartshorneAlgebraicGeometry1977} II Proposition 8.3 we have an exact sequence
	\begin{equation}
		q^* \Omega_{K / K^\tau} \rightarrow \Omega_{X/K^\tau} \rightarrow \Omega_{X/K} \rightarrow 0.
	\end{equation}
	Since $\tau : K \rightarrow K$ is finite and separable, we have $\Omega_{K / K^\tau} = 0$.
	Thus, $\Omega_{X^{(\tau)}} \simeq \Omega_X$ and the result follows.
\end{proof}

Recall that by Hodge theory we have that if $X$ is a complex projective variety, then
\begin{equation}
	H^k (X, \C) = \oplus_{p+q = k} H^{p,q}(X)
\end{equation}
where
\begin{equation}
	H^{p,q} (X) = H^p (X, \Omega_X^{q}).
\end{equation}

\begin{cor}
	\label{cor:}
	Let $X$ be a smooth complex projective variety and let $\tau : \C \rightarrow \C$ be a field automorphism, then
	the preceding identification of differential forms induces canonical $\tau$-semilinear isomorphisms
	\begin{equation}
		H^{p,q} (X^{(\tau)}) \simeq H^{p,q} (X).
	\end{equation}
\end{cor}

\begin{dfn}
	If $X,X'$ are projective varieties over a field $K$ and $f_\tau : X \rightarrow X'$ is a twisted morphism, then
	we define for every $m \geq 0$
	\begin{equation}
		(f_\tau)_* : Z_m(X) \rightarrow Z_m(X')
	\end{equation}
	by using the induced morphism of $K$-varieties $f : X^{(\tau)} \rightarrow X'$ and the canonical isomorphism
	$Z_m(X^{(\tau)}) \simeq Z_m(X)$.
	This is well defined by Lemma \ref{lemme:changing-structure-morphism}.
	Equivalently, if $X$ and $X'$ are equidimensional of dimensions $d$ and $d'$, this map sends codimension-$k$ cycles on $X$ to codimension $k+d'-d$ cycles on $X'$.
	We also have a pullback operator on Cartier divisors defined by
	\begin{equation}
		\Pic (X') \xrightarrow{f^*} \Pic\left(X^{(\tau)}\right) \xrightarrow{\sim} \Pic(X).
	\end{equation}

	If $K= \C$, then we also have $f_\tau^* : H^{p,q} (X') \rightarrow H^{p,q} (X)$ defined by
	\begin{equation}
		f_\tau^* : \Omega_{X'} \xrightarrow{f^*} \Omega_{X^{(\tau)}} \simeq \Omega_X.
	\end{equation}
\end{dfn}

\begin{lemme}
	\label{lemme:numerical-classes-changing-structure}
	Let $X$ be a $K$-variety and $\tau : \spec K \rightarrow \spec K$ a finite map, then the canonical isomorphisms
	$A_k (X) \simeq A_k (X^{(\tau)})$ yield canonical isomorphisms
	\begin{equation}
		N_k (X) \simeq N_k (X^{(\tau)})
	\end{equation}
	even though the degree function differs.
\end{lemme}
\begin{proof}
	Recall that by Lemma \ref{lemme:twist-degree-map}, $\deg^{(\tau)} = \deg(\tau) \cdot \deg$.
	Let $\alpha \in Z_k (X)$ be a $k$-cycle numerically equivalent to zero.
	We show that $\alpha^{(\tau)}$ is also numerically equivalent to zero.
	Let $Y$ be a $K$-variety with a flat morphism $q : Y \rightarrow X^{(\tau)}$ of relative dimension $e$ and $D_1, \dots, D_{e+k}$ Cartier divisors over $Y$.
	By Lemma \ref{lemme:change-of-base}, we have that $Y = Z^{(\tau)}$ and $q = g^{(\tau)}$ where $g : Z \rightarrow X$ is also flat and $D_i = E_i^{(\tau)}$ where $E_i$ are Cartier divisors over $Z$ by Lemma \ref{lemme:changing-structure-morphism}.
	Now, by Lemma \ref{lemme:twist-degree-map}
	\begin{equation}
		D_1 \dots D_{e+k} \cdot q^* \alpha^{(\tau)} = E_1^{(\tau)} \cdots E_{e+k}^{(\tau)} \cdot (g^{(\tau)})^* \alpha^{(\tau)} = (\deg \tau) \cdot E_1 \cdots E_{e + k} \cdot g^* \alpha = 0.
	\end{equation}

	Conversely if $\alpha^{(\tau)} \in Z_k(X^{(\tau)})$ is a $k$-cycle numerically equivalent to zero, then if $q : Y \rightarrow X$ is a flat
	morphism and $D_1, \dots, D_{e+k}$ are Cartier divisors over $Y$, then the induced morphism $q^{(\tau)} : Y^{(\tau)}
		\rightarrow X^{(\tau)}$ is also flat and again by Lemma \ref{lemme:twist-degree-map}
	\begin{equation}
		D_1 \dots D_{e+k} \cdot q^* \alpha = \frac{1}{\deg \tau} D_1^{(\tau)} \dots D_{e+k}^{(\tau)} \cdot (q^{(\tau)})^* \alpha^{(\tau)} = 0.
	\end{equation}
\end{proof}

\begin{dfn}[\cite{dangDegreesIteratesRational2020}]\label{dfn:pullback}
	If $\alpha \in N^k (Y)_{\mathbb R}$ and $f_\tau : X \rightarrow Y$ is a dominant twisted map, we define the pullback $f_\tau^*
		\alpha$ by duality:
	\begin{equation}
		\forall \beta \in N_k(X)_{\mathbb R}, \quad f_\tau^* \alpha \cdot \beta = \alpha \cdot (f_\tau)_* \beta.
	\end{equation}
\end{dfn}
\begin{lemme}
	\label{lemme:ampleness}
	Let $X$ be a projective variety over a field $K$ and let $\tau : \spec K \rightarrow \spec K$ be a finite map.
	If $L$ is a Cartier divisor, then $L$ is very ample (respectively ample) over $X$ if and only if $L^{(\tau)}$ is very ample (respectively ample) over $X^{(\tau)}$.
\end{lemme}
\begin{proof}
	We use the following intrinsic criterion for very ampleness: a line bundle $L$ on a projective variety is very ample if and only if, for every zero-dimensional closed subscheme $Z\subset X$ of length two, the restriction map \[ H^0(X,L)\longrightarrow H^0(Z,L_{|Z}) \] is surjective.
	This criterion depends only on the underlying scheme, the line bundle, and the restriction maps.
	These data are unchanged when we replace $X$ by $X^{(\tau)}$.
	Hence $L$ is very ample over $X$ if and only if $L^{(\tau)}$ is very ample over $X^{(\tau)}$.

	The assertion for ampleness follows immediately from the very ample case, since $(mL)^{(\tau)}=mL^{(\tau)}$ for every $m>0$.
\end{proof}

\begin{lemme}
	\label{lemme:basepoint-free-classes}
	If $X$ is a normal projective variety over a field $K$ and $\tau : \spec K \rightarrow \spec K$ is a finite map,
	then under the canonical isomorphism of real dual numerical spaces $N^k (X)_{\mathbb R} \simeq N^k (X^{(\tau)})_{\mathbb R}$ we have
	\begin{equation}
		\BPF^k (X) \simeq \BPF^k (X^{(\tau)}).
	\end{equation}
\end{lemme}
\begin{proof}
	It is enough to check the generators in the above definition.
	Twisting preserves flat morphisms and equidimensionality, and by Lemma \ref{lemme:ampleness} it preserves ample Cartier divisors.
	Hence a complete intersection class on a flat cover of $X$ is carried to the corresponding complete intersection class on the twisted flat cover of $X^{(\tau)}$.
	Since the canonical isomorphism $N^\bullet(X)_{\mathbb R}\simeq N^\bullet(X^{(\tau)})_{\mathbb R}$ is compatible with the product in the dual numerical ring, it identifies the two BPF cones.
\end{proof}

\begin{prop}
	\label{prop:positivity-twist}
	Let $X$ be a normal projective variety over a field $K$ and let $\tau:\spec K\to\spec K$ be a finite map.
	Under the canonical identifications \[ N_m(X)_{\mathbb R}\simeq N_m(X^{(\tau)})_{\mathbb R}\qquad\text{and}\qquad N^k(X)_{\mathbb R}\simeq N^k(X^{(\tau)})_{\mathbb R}, \] the following positivity cones are identified: \[ \operatorname{Eff}_m(X)\simeq \operatorname{Eff}_m(X^{(\tau)}),\qquad \operatorname{Psef}_m(X)\simeq \operatorname{Psef}_m(X^{(\tau)}), \] and, in codimension one, \[ \operatorname{Psef}^1(X)\simeq \operatorname{Psef}^1(X^{(\tau)}),\qquad \Nef^1(X)\simeq \Nef^1(X^{(\tau)}),\qquad \operatorname{Big}^1(X)\simeq \operatorname{Big}^1(X^{(\tau)}).
	\]
	Moreover, the ample cones $\operatorname{Amp}^1$ and the BPF cones $\BPF^k$ are preserved.
\end{prop}
\begin{proof}
	Effective cycles are unchanged as formal sums of integral closed subschemes when we change the structural morphism from $X$ to $X^{(\tau)}$.
	Hence the effective cones in $N_m(\cdot)_{\mathbb R}$ are identified, and the same holds for their closures, namely the pseudoeffective cones.

	For Cartier divisors, Lemma \ref{lemme:ampleness} gives preservation of ample classes.
	If $D$ is a Cartier divisor and $C$ is an integral curve, then Lemma \ref{lemme:twist-degree-map} applied to the zero-cycle $D\cdot C$ gives \[ D^{(\tau)}\cdot C^{(\tau)}=\deg(\tau)\,(D\cdot C).
	\]
	Since $\deg(\tau)>0$, non-negativity of intersections with curves is preserved.
	Thus the nef cones are identified.

	Finally, bigness is preserved by the Kodaira lemma: a real Cartier divisor class is big if and only if it can be written as the sum of an ample class and an effective class.
	Both summands are preserved by the preceding paragraphs, so the big cones are identified.
	The assertion for BPF cones is Lemma \ref{lemme:basepoint-free-classes}.
\end{proof}

\begin{prop}
	\label{prop:action-on-cycles}
	Let $f_\tau : X \rightarrow Y$ be a dominant twisted morphism between normal projective varieties.
	Then $f_\tau^*:N^\bullet(Y)_{\mathbb R}\rightarrow N^\bullet(X)_{\mathbb R}$ is a graded ring homomorphism.
	Moreover, for every
	$\alpha\in N^i(Y)_{\mathbb R}$ and every $z\in N_m(X)_{\mathbb R}$, we have the projection formula
	\begin{equation}
		(f_\tau)_*\bigl(f_\tau^*\alpha\cdot z\bigr)=\alpha\cdot (f_\tau)_*z.
		\label{eq:twisted-projection-formula}
	\end{equation}
	In particular, $f_\tau^*$ preserves BPF classes.
	If $X$ and $Y$ have the same dimension and $f_\tau$ is
	generically finite, then for any $\gamma\in N^{\dim Y}(Y)_{\mathbb R}$ one has
	\begin{equation}
		\deg\bigl(f_\tau^*\gamma\cdot [X]\bigr)=\deg(f_\tau)\deg\bigl(\gamma\cdot [Y]\bigr),\qquad
		\deg (f_\tau) := \frac{\left[ K(X) : f_\tau^* K (Y) \right]}{\left[ K : \tau^*(K) \right]}.
	\end{equation}
\end{prop}
\begin{proof}
	Let $f:X^{(\tau)}\rightarrow Y$ be the morphism of $K$-varieties induced by $f_\tau$.
	The statements for $f$ are exactly the functoriality and projection formula for dual numerical classes in \cite[Proposition 2.3.1 and Theorem 3.3.3]{dangDegreesIteratesRational2020}.
	Transporting them through the canonical identifications $N^\bullet(X^{(\tau)})\simeq N^\bullet(X)$ and $N_\bullet(X^{(\tau)})\simeq N_\bullet(X)$ gives the asserted formulas for $f_\tau$.
	The final equality is the projection formula applied to the top-degree class $\gamma$ and the fundamental class of $X$, together with the definition of the twisted degree.
\end{proof}

\begin{prop}[Toric-invariant cycles under twisting]
	\label{prop:trivial-action-toric-varieties}
	Let $X_\Sigma$ be a projective split toric variety over a field $K$, let $\tau:\spec K\to\spec K$ be a finite map, and let $\hat\tau:X_\Sigma\to X_\Sigma$ be the twisted endomorphism from Proposition \ref{prop:toric-varieties}.
	For every cone $\sigma\in\Sigma$, the canonical identification between $X_\Sigma$ and $X_\Sigma^{(\tau)}$ sends the torus-invariant orbit closure $V(\sigma)\subset X_\Sigma$ to the orbit closure corresponding to the same cone $\sigma$ in $X_\Sigma^{(\tau)}$.
	Consequently, for every $k$, the induced isomorphism
	\[
		Z^k(X_\Sigma)\simeq Z^k(X_\Sigma^{(\tau)})
	\]
	preserves the subgroup generated by torus-invariant codimension $k$ cycles, and the same is true after passing to Chow groups and numerical cycle groups.
	In particular, the action of $\hat\tau^*$ on $N^1(X_\Sigma)_\R$ is trivial:
	\[
		\hat\tau^*=\id_{N^1(X_\Sigma)_\R}.
	\]
\end{prop}
\begin{proof}
	By Proposition \ref{prop:toric-varieties}, the twist $X_\Sigma^{(\tau)}$ is described by the same lattice, the same fan, and the same face relations.
	Hence each cone $\sigma\in\Sigma$ defines the same torus orbit closure before and after twisting.
	This proves the assertion for torus-invariant cycles in $Z^k$ for every $k$, and the assertions for Chow and numerical cycle groups follow by taking the images of these invariant cycle subgroups under the natural quotient maps.

	For $k=1$, the torus-invariant prime divisors are exactly the divisors $D_\rho=V(\rho)$ corresponding to the rays $\rho\in\Sigma(1)$, and each of them is identified with the divisor corresponding to the same ray after twisting.
	Since the numerical divisor group of a projective toric variety is generated by torus-invariant divisor classes, $\hat\tau^*$ acts trivially on $N^1(X_\Sigma)_\R$.
\end{proof}
\subsection{Graphs and degrees}\label{subsec:graphs-and-degrees}

\begin{dfn}
	let $f_\tau : X \rightarrow Y$ be twisted morphism of $K$-varieties, we define the graph $\Gamma_{f_\tau}$ of $f_\tau$
	by the following procedure.
	Consider the $K$-morphism $f : X^{(\tau)} \rightarrow Y$ and its graph $\Gamma_f$.
	It is a $K$-variety and $\pi_1 : \Gamma_f \rightarrow X^{(\tau)}$ is a morphism of $K$-varieties.
	By Lemma \ref{lemme:change-of-base}, we define $\Gamma_{f_\tau}$ to be the $K$-variety
	with underlying scheme $\Gamma_f$ such that
	\begin{equation}
		\Gamma_f = \Gamma_{f_\tau}^{(\tau)}.
	\end{equation}
\end{dfn}
It is also equipped with two projections $\pi_1 : \Gamma_{f_\tau} \rightarrow X$ and $\pi_{2} = \pi_{2, \tau} : \Gamma_{f_\tau} \rightarrow Y$.
Notice that $\pi_1$ is a morphism of $K$-varieties but $\pi_2$ is a twisted morphism.

As in \cite{dangDegreesIteratesRational2020} we define the following quantities.
If $f_\tau : X \rightarrow Y$ is a
dominant twisted morphism of varieties of same dimension and $H_X, H_Y$ are big and nef divisors over $X$ and $Y$, then
\begin{equation}
	\deg_{H_X, H_Y, k} (f_\tau) := \pi_1^* H_X^{d-k} \cdot \pi_2^* H_Y^k
\end{equation}
where $\pi_1, \pi_2$ are the two projections from the graph of $f_\tau$. Notice that if $f: X^{(\tau)} \rightarrow Y$ is
the associated morphism then we have
\begin{equation}
	\deg_{H_X^{(\tau)}, H_Y, k} (f) = \frac{1}{\deg \tau} \deg_{H_X, H_Y, k} (f_\tau).
	\label{eq:relation-degrees}
\end{equation}

\begin{prop}
	\label{prop:inegalite-composition-degrees}
	Let $X,Y,Z$ be projective varieties of same dimension with dominant twisted morphisms $X \xrightarrow{f_\tau} Y \xrightarrow{g_\omega} Z$, let $H_X, H_Y, H_Z$ be big and nef divisors over $X,Y,Z$ respectively which we use to compute the different degrees.
	There exists a constant $C > 0$ depending only on $H_Y$ and the dimension $d$
	such that
	\begin{equation}
		\deg_k (g_\omega \circ f_\tau) \leq C \deg_k (f_\tau) \deg_k (g_\omega)
	\end{equation}
	Furthermore, the constant $C = \frac{(d-k+1)^k }{H_Y^d}$ works.
\end{prop}
\begin{proof}
	Let $f:X^{(\tau)}\rightarrow Y$ and $g:Y^{(\omega)}\rightarrow Z$ be the morphisms of $K$-varieties induced by $f_\tau$ and $g_\omega$.
	Twisting $f$ by $\omega$ gives a morphism of $K$-varieties
	\begin{equation}
		f^{(\omega)}:X^{(\omega\circ\tau)}\rightarrow Y^{(\omega)}.
		\label{eq:twisted-composition-ordinary-f}
	\end{equation}
	The morphism of $K$-varieties induced by the twisted composition $g_\omega\circ f_\tau$ is therefore
	\begin{equation}
		g\circ f^{(\omega)}:X^{(\omega\circ\tau)}\rightarrow Z.
		\label{eq:ordinary-composite-for-twisted-composite}
	\end{equation}

	We apply the composition inequality of \cite{dangDegreesIteratesRational2020} to \[ X^{(\omega\circ\tau)} \xrightarrow{f^{(\omega)}} Y^{(\omega)} \xrightarrow{g} Z \] with polarizations $H_X^{(\omega\circ\tau)}$, $H_Y^{(\omega)}$, and $H_Z$.
	This gives
	\begin{equation}
		\deg_k(g\circ f^{(\omega)})\leq
		\frac{(d-k+1)^k}{(H_Y^{(\omega)})^d}\,
		\deg_k(f^{(\omega)})\,\deg_k(g).
		\label{eq:ordinary-composition-inequality-twisted-proof}
	\end{equation}
	By Lemma \ref{lemme:twist-degree-map} and \eqref{eq:relation-degrees}, we have
	\begin{align}
		\deg_k(g\circ f^{(\omega)})
		 & = \frac{1}{\deg(\tau)\deg(\omega)}\deg_k(g_\omega\circ f_\tau), \\
		\deg_k(f^{(\omega)})
		 & = \deg(\omega)\deg_k(f)
		= \frac{\deg(\omega)}{\deg(\tau)}\deg_k(f_\tau),                   \\
		\deg_k(g)
		 & = \frac{1}{\deg(\omega)}\deg_k(g_\omega),
		\label{eq:degree-comparisons-composition-proof}
	\end{align}
	and also
	\begin{equation}
		(H_Y^{(\omega)})^d=\deg(\omega)H_Y^d.
		\label{eq:middle-polarization-degree-scale}
	\end{equation}
	Substituting these four identities into \eqref{eq:ordinary-composition-inequality-twisted-proof} and multiplying by $\deg(\tau)\deg(\omega)$ gives \[ \deg_k(g_\omega\circ f_\tau) \leq \frac{(d-k+1)^k}{H_Y^d}\deg_k(f_\tau)\deg_k(g_\omega), \] which is the desired estimate.
\end{proof}

\subsection{Base change}\label{subsec:base-change}
The notion of base change for twisted rational map is a bit tricky to define for the following reason.
Let $\tau : K \rightarrow K$ be a finite extension and let $\phi: K \hookrightarrow L$ be a field extension.
Then, there exists a
finite extension $\omega : L \hookrightarrow L'$ and a field homomorphism $\phi ': K \rightarrow L'$ such that the following
diagram
\begin{equation}
	\begin{tikzcd}
		K \ar[r, "\phi'"] & L' \\ K \ar[u, "\tau"] \ar[r, "\phi"] & L \ar[u, "\omega"]
	\end{tikzcd}
\end{equation}
commutes.
However, it is not true in general that we can take $L' = L$.
We give a counterexample.
\begin{ex}
	\label{ex:base-change-problem}
	Let $K = \Q(t)$ and let $L = \Q(t, \sqrt{t +1})$ with $\phi : K \rightarrow L$ the inclusion and let $\tau : K \rightarrow K$ be given by $\tau (t) = t^2$.
	Suppose we could have that $L' = L$, then since $\phi \circ \tau = \omega \circ \phi$, we must have $\omega
		(t) = t^2$, but then
	\begin{equation}
		\omega \left( (\sqrt{t +1})^2 \right) = \omega (t + 1) = t^2 +1.
	\end{equation}
	So $\omega \left(\sqrt{t+1} \right)$ must be a square root of $t^2 + 1$ in $L$ but such an element does not exist in $L$.
	The issue here is that the field extension $K \xrightarrow{\tau} K \hookrightarrow L$ is not normal.
\end{ex}

Let $f_\tau:X\dashrightarrow Y$ be a twisted rational map of $K$-varieties, and let $f:X^{(\tau)}\dashrightarrow Y$ be the corresponding rational map of $K$-varieties.
Let \[ \kappa:\spec L\longrightarrow \spec K \] be a morphism of schemes, where $L$ is algebraically closed.
We write \[ \kappa^*:K\hookrightarrow L \] for the corresponding field embedding.
Similarly, we write $\tau^*:K\hookrightarrow K$ for the field embedding corresponding to the finite morphism $\tau:\spec K\to\spec K$.

Choose an automorphism  $\sigma:\spec L\to\spec L$ such that, on functions,
\begin{equation}
	\sigma^*\circ\kappa^*=\kappa^*\circ\tau^*.
	\label{eq:field-automorphism-lifting-tau}
\end{equation}
Equivalently, the following diagram of schemes is commutative:
\begin{equation}
	\begin{tikzcd}
		\spec L \ar[r, "\sigma"] \ar[d, "\kappa"'] & \spec L \ar[d, "\kappa"] \\
		\spec K \ar[r, "\tau"'] & \spec K .
	\end{tikzcd}
	\label{eq:base-change-twist-square}
\end{equation}
Set $X_L:=X\times_{\spec K,\kappa}\spec L$ and $Y_L:=Y\times_{\spec K,\kappa}\spec L$.
The square
\eqref{eq:base-change-twist-square} gives a canonical identification
\begin{equation}
	(X_L)^{(\sigma)}\simeq X^{(\tau)}\times_{\spec K,\kappa}\spec L.
	\label{eq:twisted-base-change-domain-identification}
\end{equation}
After this identification, the ordinary base change of
$f:X^{(\tau)}\dashrightarrow Y$ gives a rational map of $L$-varieties
\begin{equation}
	f_{L,\sigma}:(X_L)^{(\sigma)}\dashrightarrow Y_L.
	\label{eq:base-changed-twisted-map}
\end{equation}
Equivalently, this is a twisted rational map \[ (f_\tau)_{L,\sigma}:X_L\dashrightarrow Y_L \] with twisting morphism $\sigma:\spec L\to\spec L$.
In diagrammatic form, this means that the following diagram commutes as a diagram of rational maps:
\begin{equation}
	\begin{tikzcd}
		X_L \ar[dd] \ar[r, "\overline q_X"] \ar[rrd, dashed, "{(f_\tau)_{L,\sigma}}"'] & \spec L \ar[rr, "\sigma"] \ar[dd, "\kappa"] & & \spec L \ar[dd, "\kappa"] \\ & & Y_L \ar[dd] \ar[ru, "\overline q_Y"'] & \\ X \ar[r, "q_X"] \ar[rrd, dashed, "f_\tau"'] & \spec K \ar[rr, "\tau"] & & \spec K \\ & & Y \ar[ru, "q_Y"'] &
	\end{tikzcd}
	\label{eq:base-changed-twisted-map-diagram}
\end{equation}
Here $\overline q_X$ and $\overline q_Y$ are the structural morphisms of $X_L$ and $Y_L$ over $\spec L$.
The top dashed arrow is the underlying rational map of the twisted map $(f_\tau)_{L,\sigma}$; equivalently, it is the ordinary rational map $f_{L,\sigma}:(X_L)^{(\sigma)}\dashrightarrow Y_L$ in \eqref{eq:base-changed-twisted-map}.

The construction depends on the choice of $\sigma$.
If $\sigma'$ is another choice satisfying $(\sigma')^*\circ\kappa^*=\kappa^*\circ\tau^*$, then $(\sigma')^*\circ(\sigma^*)^{-1}$ fixes the subfield $\kappa^*(\tau^*(K))\subset L$, and the two resulting base changes are conjugate by the corresponding automorphism of the scalar extension.
Thus the base change is not canonical, but it is well defined up to this natural conjugacy.

\begin{rmq}
	\label{rmq:base-change-extension-condition}
	Let $f_\tau:X\dashrightarrow Y$ be a twisted rational map over $K$, and let $\kappa:\spec L\to\spec K$ be a field extension.
	Put $E=\kappa^*(K)\subset L$.
	The twist $\tau$ induces an embedding
	\[
		\theta:E\longrightarrow L,\qquad
		\theta(\kappa^*(a))=\kappa^*(\tau^*a).
		\label{eq:base-change-extension-condition}
	\]
	A base change of $f_\tau$ over $L$ exists precisely when $\theta$ extends to a finite endomorphism $\sigma^*:L\to L$, or equivalently when
	\[
		\sigma^*\circ\kappa^*=\kappa^*\circ\tau^*.
	\]
	In that case the base-changed map is a twisted rational map
	\[
		(f_\tau)_{L,\sigma}:X_L\dashrightarrow Y_L
		\]
			with twisting morphism $\sigma:\spec L\to\spec L$.
			If $L$ is algebraically closed, such a finite endomorphism is an automorphism.
		\end{rmq}

\section{Dynamical degrees}
Let $f_\tau : X \dashrightarrow X$ be a dominant twisted rational selfmap of a normal projective variety $X$ of dimension $d$ over $K$,
then for any big and nef divisor $H$ over $X$, define for every $k = 0, \dots, \dim X$
\begin{equation}
	\deg_{H,k} (f_\tau) := \left( \pi_1^* H^{d-k} \cdot \pi_2^* H^k \right)
\end{equation}
where $\pi_1, \pi_2$ are the two projections from the normalisation $\tilde \Gamma_{f_\tau}$ of the graph
$\Gamma_{f_\tau}$. Here we use the canonical isomorphism $N^1(X)_{\mathbb R}\simeq
		N^1(X^{(\tau)})_{\mathbb R}$ and Proposition \ref{prop:positivity-twist} to consider $H$ also as a big and nef class
	over $X^{(\tau)}$.

\begin{rmq}
	\label{rmq:pullback-to-higher-model}
	Notice here that if $\pi : W \rightarrow \Gamma_{f_\tau}$ is a birational morphism then we can also pull back to $W$ and compute the intersection number there.
\end{rmq}

\begin{prop}
	The sequence
	\begin{equation}
		\deg_{H,k} (f_\tau^n)^{1/n}
	\end{equation}
	converges when $n \rightarrow +\infty$ towards a real number $\lambda_k (f_\tau) \geq 1$ that does not depend on $H$.
	We call it the $k$-th dynamical degree of $f_\tau$.
\end{prop}
\begin{proof}
	The proof is the same as the proof of \cite[Theorem 1]{dangDegreesIteratesRational2020}.
	In the present twisted setting, Proposition \ref{prop:positivity-twist} allows us to compare the big and nef class $H$ with its twists, and Proposition \ref{prop:inegalite-composition-degrees} gives the analogue of the submultiplicativity estimate in \cite[Theorem 1(i)]{dangDegreesIteratesRational2020}.
	Fekete's lemma then gives the existence of the limit.
	Finally, the comparison of degrees for different big and nef classes follows as in \cite[Theorem 1(ii)]{dangDegreesIteratesRational2020}.
\end{proof}

In particular, using the Khovanskii-Teissier inequalities we also get the log-concavity of the sequence $(\lambda_k
	(f_\tau))$:
\begin{equation}
	\forall k = 1, \dots, d-1, \quad \lambda_k (f_\tau)^2 \geq \lambda_{k-1} (f_\tau) \lambda_{k+1} (f_\tau).
\end{equation}

	\begin{prop}
		\label{prop:dynamical-degree-projective-space}
		If $\phi_\tau : \P^n_K \rightarrow \P^n_K$ is a dominant twisted endomorphism, then
		\begin{equation}
			\lambda_k (\phi_\tau) = \lambda_k (\phi)
		\end{equation}
		where $\phi : \P^n_K \rightarrow \P^n_K$ is the morphism of $K$-varieties such that $\phi_\tau = \hat \tau \circ
			\phi$.
		\end{prop}
	
		\begin{proof}
		Let $H=\OO_{\P^n_K}(1)$, and let $d$ be the algebraic degree of the morphism $\phi$, so that
		\[
			\phi^*H=dH .
		\]
		The twisted automorphism $\hat\tau$ preserves the hyperplane class, because it only twists the coefficients of homogeneous equations; hence
		\[
			\hat\tau^*H=H.
		\]
		Therefore
		\[
			\phi_\tau^*H=\phi^*\hat\tau^*H=dH.
		\]
		Now consider the twisted iterates of $\phi_\tau$.
		Since $\phi_\tau$ is a regular twisted endomorphism, pullback is compatible with composition, and each factor in the twisted composition pulls back $H$ to $dH$.
		By induction, for every $m\geq 1$ we get
		\[
			(\phi_\tau^m)^*H=d^mH.
		\]
		It follows that, for every $k=0,\dots,n$,
		\[
			\deg_{H,k}(\phi_\tau^m)=d^{km}.
		\]
		Thus
		\[
			\lambda_k(\phi_\tau)=d^k=\lambda_k(\phi).
		\]
	\end{proof}

\begin{prop}
	\label{prop:dyn-degrees-base-change}
	Let $X$ be a normal projective variety over $K$ and let $f_\tau : X \dashrightarrow X$ be a dominant twisted self
	rational map, then
	\begin{enumerate}
		\item If $K \hookrightarrow K '$ is a field extension where $K'$ is algebraically closed and $f_{\tau} ' : X_{K '}
			      \dashrightarrow X_{K '}$ is a base change, then $\lambda_k (f_\tau ') = \lambda_k (f_\tau)$.
		\item If $\phi_\omega : X \dashrightarrow Y$ is a birational twisted map of $K$-varieties, then
		      \begin{equation}
			      \lambda_k (f_\tau) = \lambda_k (\phi_{\omega} \circ f_\tau \circ \phi_\omega^{-1})
		      \end{equation}
	\end{enumerate}
	\end{prop}
				\begin{proof}
		For the first assertion, choose a base change of $f_\tau$ over $K'$ using an automorphism
		$\sigma:\spec K'\to\spec K'$ as in \eqref{eq:base-change-twist-square}.
		If $H$ is a big and nef divisor class on $X$ and $H_{K'}$ is its pullback to $X_{K'}$, then base change preserves big and nef classes and preserves the intersection numbers appearing in the definition of degrees.
	Therefore, for every $m\geq 1$ and every $k$,
	\[
		\deg_{H_{K'},k}\bigl((f_\tau)_{K',\sigma}^m\bigr)
		=
		\deg_{H,k}(f_\tau^m).
	\]
	Taking $m$-th roots and passing to the limit gives the desired equality.
	A different choice of $\sigma$ gives a conjugate base change, as explained above, and hence gives the same dynamical degrees.

	For the second assertion, assume that $\phi_\omega:X\dashrightarrow Y$ is a birational twisted map and that $\phi_\omega^{-1}$ denotes its twisted inverse, so that the twists in
	\[
		g:=\phi_\omega\circ f_\tau\circ\phi_\omega^{-1}
	\]
	are compatible.
	Then, for every $m\geq 1$,
	\[
		g^m=\phi_\omega\circ f_\tau^m\circ\phi_\omega^{-1}
		\]
		as twisted rational maps.
		After passing to the associated ordinary rational maps between the corresponding twists of $X$ and $Y$, this is the usual birational conjugacy situation.
		Thus the equality of dynamical degrees follows from the birational invariance theorem \cite[Theorem 1]{dangDegreesIteratesRational2020}.
		\end{proof}

	 We now turn to computing dynamical degrees for algebraic tori.

		\begin{prop}
			\label{prop:absolute-dyn-degrees-algebraic-tori}
			Let $d$ be an integer and $K$ be a field and let $\phi_\tau : \G_m^d \rightarrow \G_m^d$, then for any $k =1, \dots, d$,
			\begin{equation}
				\lambda_k (\phi_\tau) = \lambda_k(\phi)
			\end{equation}
			where $\phi_\tau = \hat \tau \circ \phi$ and $\phi : \G_m^d \rightarrow \G_m^d$ is the regular endomorphism coming from Corollary \ref{cor:twisted-endo-alg-torus}.
			{In particular, $\lambda_1(\phi_\tau)$ is an algebraic integer of degree $\leq d$.}
		\end{prop}
	\begin{proof}
		By Proposition \ref{prop:dyn-degrees-base-change}, we may assume that $K$ is algebraically closed.
	Then $\tau$ is an automorphism of $K$.
	Write
	\[
		\phi=t_b\circ \phi_A,
	\]
	where $b\in (K^*)^d$, $t_b$ is translation by $b$, and $A=(a_{ij})\in M_d(\Z)$ is the exponent matrix, so that
	\[
		\phi_A(x_1,\dots,x_d)=
		\left(\prod_j x_j^{a_{1j}},\dots,\prod_j x_j^{a_{dj}}\right).
	\]
	The twisted map $\phi_\tau=\hat\tau\circ\phi$ has the same exponent matrix $A$; the twist only affects the coefficient part under iteration.
	More precisely, for every $m\geq 1$, there exists a point $b_m\in (K^*)^d$ such that the $m$-th twisted iterate has the form
	\[
		\phi_\tau^m=\widehat{\tau^m}\circ t_{b_m}\circ \phi_{A^m}
	\]
	as a twisted rational map.
	Indeed, each composition only changes the coefficient vector, by multiplying it with translates of its images under powers of $\tau$, while the exponent matrix is multiplied in the usual way.
		We now compare degrees before taking the limit.
		Let $X=(\P^1)^d$ and let
		\[
			H=\sum_{i=1}^d \pr_i^*\OO_{\P^1}(1).
		\]
			The twisted map $\phi_\tau^m$ corresponds to the ordinary rational map
			\[
				F_m:=t_{b_m}\circ \phi_{A^m}:X^{(\tau^m)}\dashrightarrow X.
			\]
			Since $K$ is algebraically closed, $\tau^m$ is an automorphism and $\deg(\tau^m)=1$.
			Moreover, under the canonical identification $N^1(X^{(\tau^m)})_{\mathbb R}\simeq N^1(X)_{\mathbb R}$, the class $H^{(\tau^m)}$ is identified with $H$.
			Thus the definition of the degree of a twisted rational map gives
			\[
				\deg_{H,k}(\phi_\tau^m)
				=
				\deg_{H^{(\tau^m)},H,k}(F_m)
				=
				\deg_{H,k}(t_{b_m}\circ \phi_{A^m}).
			\]
			The translation $t_{b_m}$ extends to an automorphism of $X$ and preserves the numerical class of $H$.
			Hence, for every $m\geq 1$,
			\[
				\deg_{H,k}(\phi_\tau^m)
				=
				\deg_{H,k}(\phi_{A^m}).
			\]
		By the standard formula for the degree growth of monomial maps \cite{favreDegreeGrowthMonomial2012,linPullingBackCohomology2012},
		\[
			\lim_{m\to\infty}\deg_{H,k}(\phi_{A^m})^{1/m}
			=
			\rho\left(\wedge^k A\right).
		\]
		Thus
		\[
			\lambda_k(\phi_\tau)=\rho\left(\wedge^k A\right).
		\]
		The same degree comparison, with the coefficient vector $b_m$ replaced by the coefficient vector appearing in the ordinary iterate $\phi^m$, gives
		\[
			\lambda_k(\phi)=\rho\left(\wedge^k A\right).
		\]
		Therefore $\lambda_k(\phi_\tau)=\lambda_k(\phi)$, as claimed.
		For $k=1$, this gives $\lambda_1(\phi_\tau)=\rho(A)$.
		Since $A$ is a $d\times d$ matrix with integral coefficients, this is an algebraic integer of degree at most $d$.
	\end{proof}

\subsection{Relative dynamical degrees}\label{subsec:relative-dynamical-degrees}
Following \cite{truongRelativeDynamicalDegrees2015}, let $X$ and $Y$ be normal projective $K$-varieties, let
$q:X\rightarrow Y$ be a dominant morphism, and let
$f:X\dashrightarrow X$ and $g:Y\dashrightarrow Y$ be dominant rational selfmaps such that
\[
q\circ f=g\circ q.
\]
Set $d_X=\dim X$, $d_Y=\dim Y$, and $e=d_X-d_Y$.
For any $0\leq k\leq e$, the \emph{relative $k$-th dynamical degree of $f$} is defined as the limit
\begin{equation}
	\lambda_k (f_{|q}) = \lim_n \reldeg_k (f^n)^{1/n}
\end{equation}
where the definition of $\reldeg_k (f)$ is as follows.
Let $H_X, H_Y$ be big and nef divisors over $X$ and $Y$ respectively.
Let $\tilde \Gamma_{f^n}$ be the normalisation of the graph of $f^n$ with its two projections $\pi_1, \pi_2$ to $X$.
Then
\begin{equation}
	\reldeg_{k, H_X, H_Y} (f^n) = \left( (\pi_2^* H_X)^k \cdot (\pi_1^* H_X)^{{d_X}- d_Y - k} \cdot (\pi_1 \circ q)^*
	H_Y^{d_Y}  \right).
\end{equation}
Again the limit exists and does not depend on the choice of $H_X$ and $H_Y$ and it is a birational invariant.

\begin{thm}
	\label{thm:relative-dyn-degree-is-twisted-dynamical-degree}
	{Under the assumptions above, let $\eta$ be the generic point of $Y$.
	The map $g$ induces a finite morphism $g_\eta:\spec K(Y)\to\spec K(Y)$, and the restriction of $f$ to the generic fiber defines a twisted rational map
	$f_\eta:X_\eta\dashrightarrow X_\eta$ with twisting $g_\eta$.
	Then for any $0\leq k\leq e$,}
	\begin{equation}
		\lambda_k (f_{|q}) = \lambda_k (f_{\eta} : X_\eta \dashrightarrow X_\eta)
	\end{equation}
\end{thm}
This follows from the following result.

	\begin{prop}
		\label{prop:intersection-number-fibration}
		Let $q :X \rightarrow Y$ be a {dominant} morphism of normal projective $K$-varieties.
		Write $e = d_X - d_Y$ where $d_Z = \dim Z$ for $Z = X,Y$.
		Let {$D_1, \dots, D_e$ be Cartier $\R$-divisor classes on $X$ and $L_1, \dots, L_{d_Y}$ be Cartier $\R$-divisor classes on $Y$}, we have
		\begin{equation}
			\deg\left( L_1 \cdots L_{d_Y} \right) D_{1,\eta} \cdots D_{e, \eta} = D_1 \cdots D_e \cdot q^* L_1 \dots q^* L_{d_Y}.
		\end{equation}
	\end{prop}
	\begin{proof}
				By multilinearity, it is enough to treat Cartier divisors.
		Set
		\[
			\alpha:=D_1\cdots D_e\cap [X]\in A_{d_Y}(X).
		\]
		By definition of the intersection product on the generic fiber, the coefficient of $[Y]$ in $q_*\alpha$ is
		\[
			D_{1,\eta}\cdots D_{e,\eta}.
		\]
		Hence
		\[
			q_*\alpha=(D_{1,\eta}\cdots D_{e,\eta})[Y]
		\]
		as a top-dimensional cycle on the integral variety $Y$.
		Applying the projection formula for intersections with Cartier divisors \cite[Proposition 2.3]{fultonIntersectionTheory1998}, we get
		\[
		\begin{aligned}
			D_1\cdots D_e\cdot q^*L_1\cdots q^*L_{d_Y}
			&=
			L_1\cdots L_{d_Y}\cdot q_*\alpha \\
			&=
			(D_{1,\eta}\cdots D_{e,\eta})\deg(L_1\cdots L_{d_Y}),
		\end{aligned}
		\]
		which is the desired identity.
		The statement for $\R$-classes follows again by multilinearity.
	\end{proof}

	\begin{proof}[Proof of Theorem \ref{thm:relative-dyn-degree-is-twisted-dynamical-degree}]
	Consider the following diagram.
	\begin{equation}
		\begin{tikzcd}
			& \Gamma_{f^n}  \ar[ld, "\pi_1"] \ar[rd, "\pi_2"] \\
			X \ar[rr, dashed, "{f^n}"] \ar[d, "q"] & & X \ar[d, "q"] \\
			Y \ar[rr, dashed, "{g^n}"] & & Y \\
		\end{tikzcd}
	\end{equation}
	{Restricting this graph to the generic point $\eta$ of $Y$, i.e. tensoring with $\spec K(\eta)$, gives the graph of the twisted rational map $f_\eta^n$.
	Equivalently, any birational model of the graph obtained in this way computes the same intersection numbers.}
	\begin{equation}
		\begin{tikzcd}
			& \Gamma_{f_\eta^n}  \ar[ld, "\pi_1"] \ar[rd, "\pi_2"] \\
			X_\eta \ar[rr, dashed, "{f_\eta^n}"] \ar[d, "q_\eta"] & & X_\eta \ar[d, "q_\eta"] \\
			\spec K(Y) \ar[rr, "{(g_\eta)^n}"] & & \spec K(Y) \\
		\end{tikzcd}
	\end{equation}
	So by Proposition \ref{prop:intersection-number-fibration} we have
	\begin{align}
		\reldeg_{k, H_X, H_Y} (f^n) & = \left( (\pi_2^* H_X)^k \cdot (\pi_1^* H_X)^{{d_X}- d_Y - k} \cdot (\pi_1 \circ q)^*
		H_Y^{d_Y}  \right)                                                                                               \\
		                            & = H_Y^{d_Y}  \left( (\pi_2^* H_{X,\eta})^k \cdot (\pi_1^* H_{X,\eta})^{{d_X}- d_Y - k}
	\right)                                                                                                          \\
		                            & = H_Y^{d_Y} \deg_k (f_\eta^n).
	\end{align}
	Taking the $n$-th roots and letting $n \rightarrow + \infty$ yields the result.
\end{proof}

\section{Relative degree formula}
\subsection{Cartier-$b$-classes and Siu's inequality}\label{subsec:siu-inequality}
Let $X$ be a projective variety over a field $K$.
We define the space $\cNk (X)$ of numerical Cartier $b$-classes of codimension
$k$ as the direct limit
\begin{equation}
	\cNk(X) = \varinjlim_{Y \rightarrow X}
	N^k (Y)_{\R},
\end{equation}
where $Y$ runs over the projective birational models of $X$ and the transition maps are pullbacks.
Since every projective birational model is dominated by a normal projective one, one may equivalently take the direct limit over normal projective models.
If $K$ has characteristic zero, one may also take the direct limit over smooth projective models.
Concretely, an element in $\cNk (X)$ is an equivalence class of a tuple $(Y, \alpha)$ where $q : Y \rightarrow X $ is a projective birational morphism and $\alpha \in N^k(Y)_{\R}$ is a numerical class of cycles of codimension $k$ and $(Y, \alpha) \sim (Z, \beta)$ if there exists $W$ with two birational morphisms $\pi_Y : W \rightarrow Y, \pi_Z: W \rightarrow Z$ such that $\pi_Y^* \alpha = \pi_Z^* \beta$.
We say that $\alpha$ is \emph{determined} in $Y$.
We say that $\alpha \in \cNk(X)$ is \emph{pseudoeffective} or psef if for any $(Z, \beta) \sim \alpha$ the class $\beta$ is in the closure of the cone of effective cycles of codimension $k$, we write $\alpha \geq 0$ when this is the case and $\alpha \geq \beta$ whenever $\alpha - \beta \geq 0$.
If $k = 1$ then $\alpha \geq 0$ if and only if $\alpha$ belongs to the pseudoeffective cone of divisors of $Y$ where $\alpha$ is determined in $Y$ because the pullback of an effective divisor is still an effective divisor.

Finally we say that $\alpha \in \cNone (X)$ is \emph{nef} if there exists $\beta \sim \alpha$ determined in some $Y$ such that $\beta \in \Nef(Y)$.
We now state the generalised version of Siu's inequality and another statement on intersection of nef classes.
We say that $\alpha \in \cNone(X)$ is \emph{big} if $\alpha \in \text{Big}(Y)$ the cone of big divisors where $\alpha$ is determined in $Y$.

For $D_1, \dots, D_k \in \cNone(X)$ there is a well defined intersection product $D_1 \cdots D_k \in \cNk (X)$ defined by taking the intersection in any $Y$ where all the $D_i$'s are defined.
This does not depend on the choice of $Y$.
If
$q : X \dashrightarrow X'$ is a dominant rational map of $K$-varieties then we have a well defined pullback
homomorphism
\begin{equation}
	q^* : \cNk (X') \rightarrow \cNk (X)
\end{equation}
given by $q^* \alpha = (Y, q^* \alpha)$ where $\pi ' : Y' \rightarrow X'$ and
$\pi : Y \rightarrow X$ are birational morphisms such that the induced map $q : Y \rightarrow Y'$ is regular and
$\alpha$ is defined over $Y'$, it does not depend on the choice of $Y', X'$.

By Lemma \ref{lemme:numerical-classes-changing-structure}, if $\tau : \spec K \rightarrow \spec K$ is a finite map and $X$ is a $K$-variety then we have canonical isomorphisms $\cNk (X^{(\tau)}) \simeq \cNk (X)$.
Therefore if $q_\tau : X
	\dashrightarrow Y$ is a twisted dominant rational map then it induces a pullback homomorphism
\begin{equation}
	q_\tau^* : \cNk(Y) \rightarrow \cNk (X)
\end{equation}
given by $\cNk (Y) \xrightarrow{q^*} \cNk (X^{(\tau)}) \simeq \cNk (X)$.

In particular, the definition of the dynamical degrees of $f_\tau : X \dashrightarrow X$ can be written as follows.
If
$L \in \cNone(X)$ is big and nef then
\begin{equation}
	\deg_{k,L} (f_\tau^n) = \left((f_\tau^n)^* L\right)^k \cdot L^{d-k} \quad \text{and} \quad \lambda_k (f_\tau) = \lim_n
	\deg_{k,L} (f_\tau^n)^{1/n}.
\end{equation}

We now state two very important results for the theory of Cartier $b$-classes.

\begin{thm}[Siu's inequality, \cite{dangDegreesIteratesRational2020}]\label{thm:siu}
	Let $\alpha_1, \dots, \alpha_i \in \cNone (X)$ be nef classes and let $\beta \in \cNone (X)$ be a big and nef class,
	then there exists a constant $C = C(i,d)$ depending only on $d =\dim X$ and $i$ such that
	\begin{equation}
		\alpha_1 \cdots \alpha_i \leq C(i,d) \frac{\alpha_1 \cdots \alpha_i \cdot \beta^{d-i}}{\beta^d} \beta^i.
	\end{equation}
\end{thm}

\begin{prop}[\cite{boucksomDifferentiabilityVolumesDivisors2009}]\label{prop:inequality-intersection-nef-classes}
	Let $\alpha_j, \beta_j \in \cNone(X)$ be nef classes such that $\alpha_j \leq \beta_j$ for $j=1, \dots, i$. Then
	\begin{equation}
		\alpha_1 \cdot \alpha_2 \cdots \alpha_i \leq \beta_1 \cdot \beta_2 \cdots \beta_i.
		\label{}
	\end{equation}
\end{prop}

\subsection{A twisted mixed degree formula}\label{subsec:twisted-mixed-degree-formula}

We prove in this subsection the twisted mixed degree formula stated in Theorem \ref{bigthm:relative-dyn-degrees-formula}.
Let $f_\tau : X \dashrightarrow X$ and $g_\omega : Y \dashrightarrow Y$ be twisted rational maps such that there exists a dominant twisted rational map $q_\sigma : X \dashrightarrow Y$ of relative dimension $e \geq 0$ satisfying $q_\sigma \circ f_\tau = g_\omega \circ q_\sigma$.
By replacing $X$ with the normalization of the graph of $q_\sigma$ and using birational invariance of dynamical degrees, we may assume in the proof that $q_\sigma$ is a twisted morphism.
Fix $L, M$ big and nef divisors over $X$ and $Y$ respectively.
We write $L_n :=
	(f_\tau^n)^* L$ and $M_n = (g_\omega^n)^* M$ and we write
\begin{equation}
	\deg_i (f_\tau^n) := L_n^i \cdot L^{d-i}, \quad \deg_i (f_\eta^n) = L_{n,\eta}^i \cdot L_\eta^{d-d_Y - i}, \quad \deg_j
	(g_\omega^n) := M_n^j \cdot M^{d_Y - j}.
\end{equation}
We make an abuse of notation and still write $M$ for its pullback $q_\sigma^* M$ in $X$ this does not yield confusion.

\begin{lemme}
	\label{lemme:estimate-iteration}
  For any $n \geq 1$ and any $k \geq 0$, if $i \leq \dim X - \dim Y$, there exist a constant $C$  depending only on
    $\dim X, \dim Y, L, M$ and such that for any $a > 0$ large enough
	\begin{equation}
    L_{(k+1)m}^i \leq C \deg_i (f_\eta^m) (L_{km}+aM_{km})^i.
	\end{equation}
	If $i > \dim X - \dim Y$ then for any $0 < \rho < 1$ and any $n \geq 1$ we have for any $a > 0$ large enough
	\begin{equation}
    L_{(k+1)m}^i \leq \rho (L_{km}+aM_{km})^i.
	\end{equation}
\end{lemme}
The crucial point here is that the constant $C$ is independent of $n$  and $k$.
\begin{proof}
	We use Siu's inequality.
	Fix $a > 1$, we have
	\begin{align}
    L_{(k+1)m}^i & \leq C(i,d) \frac{L_{(k+1)m}^i \cdot (L_{km}+aM_{km})^{d-i}}{(L_{km}+aM_{km})^d} (L_{km}+aM_{km})^i \\
    &= C(i,d) \frac{L^i_m \cdot (L + aM)^{d-i}}{(L+aM)^{d}} (L_{km}+aM_{km})^i.
	\end{align}
	Notice that both numerators and denominators are polynomials in $a$ with nonnegative coefficients of degree $\leq d_Y$ which we call $N(a)$ and $D(a)$.
	Suppose first that $i \leq d - d_Y$, then $d-i \geq d_Y$ and both $N(a)$ and $D(a)$ are of degree $d_Y$.
	The leading coefficient of $N(a)$ is by Proposition \ref{prop:intersection-number-fibration} equal to
	\begin{equation}
		L_m^i \cdot L^{d-d_Y - i} \cdot M^{d_Y} = \deg_i (f_\eta^m) M^{d_Y}.
	\end{equation}
	And the leading coefficient of $D(a)$ is $L_\eta^{d-d_Y} M^{d_Y}$.
	Thus we see that
	\begin{equation}
		\frac{N(a)}{D(a)} \xrightarrow[a \rightarrow +\infty]{} \frac{\deg_i (f_\eta^m)}{L_\eta^{d-d_Y}}
	\end{equation}
	So that there exists $a_0 = a_0 (n) > 0$ such that for $a > a_0, \frac{N(a)}{D(a)} \leq \frac{2}{L_\eta^{d-d_Y}} \deg_i (f_\eta^m)$.

	If $i > d-d_Y$ then $d_0 := d-i < d_Y$ and $N(a)$ has degree $d_0$ whereas $D(a)$ still has degree $d_Y$ so that
	\begin{equation}
		\frac{N(a)}{D(a)} \xrightarrow[a \rightarrow +\infty]{} 0.
	\end{equation}
	For any $0 < \rho < 1$ there exists $a_0 = a_0 (n) > 0$ such that for any $a > a_0$,
	\begin{equation}
		C(i,d) \frac{N(a)}{D(a)} \leq \rho.
	\end{equation}
\end{proof}

\begin{cor}
	\label{cor:inequality-for-iteration-proof}
  For any $0 \leq i \leq d-d_Y$ and any $\epsilon > 0$, there exists $m_0 > 0$ such that for all $m \geq m_0$
  and any $k \geq 0$ there are constants $C_1(m), \dots, C_{\min(i,d_Y)} (m)> 0$ such that
	\begin{equation}
    L_{(k+1)m}^i \leq (\lambda_i (f_\eta) + \epsilon)^m L_{km}^i + \sum_{s = 1}^{\min (i, d_Y)} C_s(m)
    L_{km}^{i-s} M_{km}^s.
	\end{equation}
  For any $d-d_Y < i \leq d$, any $0 < \rho < 1$, any $m \geq 1$ and any $k \geq 0$ there exist constants
	$C_1 (m), \dots, C_{\min(i,d_Y)} (m)>0$ such that
	\begin{equation}
    L_{(k+1)m}^i \leq \rho L_{km}^i +\sum_{s = 1}^{\min (i, d_Y)} C_s(m) L_{km}^{i-s} M_{km}^s.
	\end{equation}
  For any $0\leq j \leq d_Y$ and any $\epsilon > 0$ there exists $m_1 > 0$ such that for $m \geq m_1$ and any
  $k \geq 0$
	\begin{equation}
    M_{(k+1)m}^j \leq (\lambda_j (g_\omega) + \epsilon)^m M_{km}^j.
	\end{equation}
\end{cor}
\begin{proof}
  We first prove the two estimates. 
	Fix $i \leq d - d_Y$ and let $\epsilon > 0$.
	Let $C$ be the constant appearing in Lemma \ref{lemme:estimate-iteration}, we reiterate that the constant does not depend on $m$ nor $k$.
	For $m \geq 1$ large enough we have
	\begin{equation}
		\deg_i (f_\eta^m) \leq (\lambda_i (f_\eta) + \epsilon / 2)^m, \quad \text{and} \quad C (\lambda_i (f_\eta) +
		\epsilon / 2)^m \leq (\lambda_i (f_\eta) + \epsilon)^m.
	\end{equation}
	For such an $m$, Lemma \ref{lemme:estimate-iteration} gives, after choosing $a=a(m)>0$ large enough,
	\[
    L_{(k+1)m}^i\leq C\deg_i(f_\eta^m)(L_{km}+aM_{km})^i.
	\]
  Expanding $(L_{km}+aM_{km})^i$ and absorbing all terms containing at least one factor of $M$ into the constants $C_s(m)$ gives the first estimate.

	If $i>d-d_Y$, Lemma \ref{lemme:estimate-iteration} gives, for the prescribed $\rho$ and for $a=a(m)>0$ large enough,
	\[
    L_{(k+1)m}^i\leq \rho (L_{km}+aM_{km})^i.
	\]
	Expanding again and absorbing the terms containing $M$ into the constants $C_k(m)$ gives the second estimate.

	It remains to prove the estimate for $M_m^j$.
	The case $j=0$ is trivial.
	For $1 \leq j \leq d_Y$, Siu's inequality gives
	\begin{equation}
    M_{(k+1)m}^j \leq C(j,d_Y) \frac{\deg_j (g_\omega^m)}{M^{d_Y}} M_{km}^j.
	\end{equation}
	For $m \geq 1$ large enough we have
	\begin{equation}
		\deg_j (g_\omega^m) \leq (\lambda_j (g_\omega) + \epsilon / 2)^m \quad \text{and} \quad C(j,d_Y) \frac{(\lambda_j
			(g_\omega) + \epsilon /2)^m}{M^{d_Y}} \leq (\lambda_j (g_\omega) + \epsilon)^m.
	\end{equation}
  This proves the desired inequality for $M_{(k+1)m}^j$.
\end{proof}

We now prove the main result of this subsection.
The following theorem is the precise form of the mixed degree formula stated in Theorem \ref{bigthm:relative-dyn-degrees-formula}.
\begin{thm}
	\label{thm:relative-degree-formula}
	Let $X,Y$ be projective varieties with $q_\sigma : X \dashrightarrow Y$ a twisted dominant rational map and $f_\tau : X
		\dashrightarrow X, g_\omega : Y \dashrightarrow Y$ twisted rational maps such that we have the commutative diagram
	\begin{equation}
		\begin{tikzcd}
			X \ar[r,dashed, "f_\tau"] \ar[d,dashed, "q_\sigma"] & X \ar[d, dashed, "q_\sigma"] \\
			Y \ar[r, dashed, "g_\omega"] & Y
		\end{tikzcd}
	\end{equation}
	Then, for every $i = 1, \dots, \dim X$,
	\begin{equation}
		\lambda_i (f_\tau) = \max_{\substack{a+b = i \\ a \leq \dim X - \dim Y \\ b \leq \dim Y}} \left( \lambda_a (f_\eta) \lambda_b (g_\omega) \right).
	\end{equation}
\end{thm}
\begin{proof}
	We first show the inequality $\leq$.
	Let $\epsilon > 0$ and fix $0 < \rho < 1$ such that for any $j = 1, \dots, d_Y$ we have $ \rho \cdot (\lambda_j (g_\omega) + \epsilon) < 1$.
	By Corollary \ref{cor:inequality-for-iteration-proof} there exists $m > 0$ such that for any $1 \leq i \leq d$
  and any $k \geq 0$
	\begin{equation}
  L_{(k+1) m}^i \leq u_i (m) L_{km}^i + \sum_{s=1}^{\min (i, d_Y)} C_s L_{km}^{i - s} \cdot M_{km}^s
	\end{equation}
	where $u_i (m) = (\lambda_i (f_\eta) + \epsilon)^m$ if $i \leq d - d_Y$ or $u_i (m) = \rho$ if $i > d-d_Y$.
	and for any $j = 1, \dots, d_Y$
	\begin{equation}
    M_{(k+1)m}^j \leq (\lambda_j (g_\omega) + \epsilon)^m M_{km}^j.
	\end{equation}
	Intersecting the two inequalities we get that for any $0 \leq i \leq d$ and $0 \leq j \leq d_Y$ we have
	\begin{equation}
    L_{(k+1)m}^i \cdot M_{(k+1)m}^j \leq u_i (m) \left( \lambda_j (g_\omega) + \epsilon \right)^m L_{km}^i \cdot M_{km}^j +
    \sum_{s = 1}^{\min (i, d_Y)} C_s L_{km}^{i-s} \cdot M_{km}^{s+j}.
		\label{eq:main-inequality-to-iterate}
	\end{equation}
	We can write these inequalities in the following form: for any $i \leq d$, let $i_0 = \min (i,d_Y)$ then
	\begin{equation}
			\begin{pmatrix}
				L_{(k+1)m}^i                               \\
				L_{(k+1)m}^{i-1} M_{(k+1)m}                       \\
				\vdots                              \\
				L_{(k+1)m}^{i - i_0 + 1} \cdot M_{(k+1)m}^{i_0-1} \\
				L_{(k+1)m}^{i-i_0} \cdot M_{(k+1)m}^{i_0}
			\end{pmatrix}
		\leq A
		\begin{pmatrix}
      L_{km}^i                             \\
      L_{km}^{i-1} M_{km}                       \\
			\vdots                          \\
      L_{km}^{i - i_0 + 1} \cdot M_{km}^{i_0-1} \\
      L_{km}^{i-i_0} \cdot M_{km}^{i_0}
		\end{pmatrix}
		\end{equation}
		with $A$ an upper triangular matrix with nonnegative coefficients and with diagonal coefficients
	\begin{equation}
		u_{i-k} (m) \left( \lambda_{k} (g_\omega) + \epsilon \right)^m
	\end{equation}
	with $k =0 ,\dots, \min (i, d_Y)$.
	We get by induction that for any $n \geq 1$
	\begin{equation}
		\begin{pmatrix}
			L_{mn}^i                                  \\
			L_{mn}^{i-1} M_{mn}                       \\
			\vdots                                    \\
			L_{mn}^{i - i_0 + 1} \cdot M_{mn}^{i_0-1} \\
			L_{mn}^{i-i_0} \cdot M_{mn}^{i_0}
		\end{pmatrix}
		\leq A^n
		\begin{pmatrix}
			L^i                             \\
			L^{i-1} M                       \\
			\vdots                          \\
			L^{i - i_0 + 1} \cdot M^{i_0-1} \\
			L^{i-i_0} \cdot M^{i_0}
		\end{pmatrix}
		\label{eq:iteration-matrix}
	\end{equation}
	The norm of the matrix $A^n$ grows like $T^n$ where
	\begin{equation}
		T : = \max_{k = 0, \dots, \min(i,d_Y)} u_{i-k}(m) (\lambda_{k} (g_\omega) + \epsilon)^{m} .
	\end{equation}
	Notice that when $i - k > d - d_Y$ we have that the term inside the maximum is $< 1$ so that
	\begin{equation}
		T = \max_{\substack{a+b = i \\ a \leq d - d_Y \\ b \leq d_Y}} (\lambda_a (f_\eta) + \epsilon)^m (\lambda_b (g_\omega) + \epsilon)^m .
	\end{equation}
	Looking at the first coordinate in \eqref{eq:iteration-matrix} we get that
	\begin{equation}
		L_{mn}^i \leq 2 T^n \left( L^i + L^{i-1} M + \cdots M^i \right)
	\end{equation}
	for $n \geq 1$ large enough.
	Intersecting with $L^{d-i}$ and taking $n$-th roots we get
		\begin{equation}
			\lambda_i (f_\tau^m) = \lambda_i (f_\tau)^m \leq
			\max_{\substack{a+b = i \\
					a \leq d - d_Y \\
					b \leq d_Y}}  (\lambda_a (f_\eta) + \epsilon)^m (\lambda_b (g_\omega) + \epsilon)^m
		\end{equation}
	which yields
		\begin{equation}
			\lambda_i (f_\tau) \leq \max_{\substack{a+b = i \\
					a \leq d - d_Y \\
					b \leq d_Y}}  (\lambda_a (f_\eta) + \epsilon) (\lambda_b (g_\omega) + \epsilon).
		\end{equation}
		Letting $\epsilon \rightarrow 0$ yields the desired inequality.

	The other inequality is easier.
	Compute $\lambda_i (f_\tau)$ using the big and nef divisor $(L + M)$, we have
	\begin{equation}
		\lambda_i (f_\tau) = \lim_n \left( (L_n + M_n)^i \cdot (L + M)^{d - i} \right)^{1/n}.
	\end{equation}
	Let $A_n$ be the intersection product inside the $n$-th root.
	There are positive constants $C_k$ depending only on $i$ and $d$ such that
	\begin{equation}
		(L_n + M_n)^i = \sum_{k = 0}^{\min (i, \dim Y)} C_k M_n^k \cdot L_n^{i-k}
	\end{equation}
	and
	\begin{equation}
		(L+M)^{d-i} = \sum_{l = 0}^{\min(\dim Y, d-i)}
		C_l M^l L^{d - i - l}.
	\end{equation}
	Intersecting these two formulas we get that there are constants $C_{a,b}$ not depending on $n$ such that
	\begin{equation}
		A_n \geq \sum_{\substack{ 0 \leq a \leq d- d_Y \\ 0 \leq b \leq d_Y \\ a+b = i}} C_{a,b} \left( L_n^{a} \cdot L^{d - d_Y - a} \right) \cdot (M_n^b \cdot M^{d_Y -b}).
	\end{equation}
		Using the fibration $q_\sigma : X \dashrightarrow Y $ and Proposition \ref{prop:intersection-number-fibration} we have
		\begin{equation}
			\left( L_n^{a} \cdot L^{d - d_Y - a} \right) \cdot (M_n^b \cdot M^{d_Y - b}) = \deg_a (f_\eta^n) \cdot \deg_b
			(g_\omega^n).
		\end{equation}
		Taking the $n$-th root and letting $n \rightarrow + \infty$, we get the reversed inequality.
\end{proof}
In the last use of Proposition \ref{prop:intersection-number-fibration}, one applies the proposition to the ordinary morphism $q:X^{(\sigma)}\to Y$ associated to the twisted morphism $q_\sigma$ and then identifies the numerical classes of $X^{(\sigma)}$ with those of $X$.
The factor $\deg(\sigma)$ appearing in intersection numbers after changing the structural morphism appears both in the total-space intersection and in the intersection number on the generic fiber, so the equality is unchanged when written in the twisted notation.
This proves the mixed degree formula in Theorem \ref{bigthm:relative-dyn-degrees-formula}.

\begin{rmq}
	\label{rmq:}
	It is equivalent to show the theorem when $q : X \dashrightarrow Y$ is a dominant rational map of $K$-varieties and $\tau = \omega$.
	Indeed, from the commutative diagram
	\begin{equation}
		\begin{tikzcd}
			X \ar[r, dashed, "f_\tau"] \ar[d, dashed, "q_\sigma"] & X \ar[d, dashed, "q_\sigma"] \\
			Y \ar[r, dashed, "g_\omega"'] & Y
		\end{tikzcd}
	\end{equation}
	we have the induced commutative diagram
	\begin{equation}
		\begin{tikzcd}
			X^{(\sigma)} \ar[r, dashed, "f_\omega"] \ar[d, dashed, "q"] & X^{(\sigma)} \ar[d, dashed, "q"] \\ Y \ar[r, dashed, "g_\omega"'] & Y.
		\end{tikzcd}
	\end{equation}
	Notice that $f_\tau$ and $f_\omega$ induce the same twisted rational map on the generic fiber.
	And we have
	\begin{equation}
		\lambda_i (f_\tau) = \lambda_i (f_\omega)
	\end{equation}
	since
	\begin{equation}
		\deg_i (f_\tau^n) = \deg(\sigma) \deg_i (f_\omega^n).
	\end{equation}
\end{rmq}

\section{Banach spaces and dynamical degrees}\label{sec:banach-spaces}
Let $X$ be a smooth projective variety over a field of characteristic zero.
In \cite{dangSpectralInterpretationsDynamical2021}, Dang and Favre defined several Banach spaces over which dominant rational selfmaps of $X$ acts naturally.
We briefly recall their construction and show how it extends naturally to twisted rational maps.
A \emph{birational model} of $X$ is a smooth projective variety $Y$ equipped with a birational morphism $\pi : Y \rightarrow X$.
Define for any $1 \leq k \leq \dim X$ the space of numerical Cartier and Weil classes
\begin{equation}
	\cNk (X) = \varinjlim_{Y \rightarrow X}
	N^k (Y)_\R, \quad \wNk (X) = \varprojlim_{Y \rightarrow X} N^k(Y)_\R.
\end{equation}
We have already discussed numerical Cartier classes in \S \ref{subsec:siu-inequality}.
Since the ground field has characteristic zero, the direct and inverse limits may be taken over smooth projective birational models.
An element of $\wNk (X)$ is the data of a family $\alpha = \left( Y, \alpha_Y \right)_Y$ where $Y$ runs through the birational models of $X$ such that if there is a birational morphism $\pi : Y \rightarrow Y'$, then $\pi_* \alpha_Y = \alpha_{Y'}$.
There is a natural injection $\cNk (X) \hookrightarrow \wNk (X)$ and we endow $\wNk (X)$ with the inverse limit topology, which we call the \emph{weak topology}.
There is a well defined intersection product
\begin{equation}
	\cNk (X) \times \wNN^{d - k} (X) \rightarrow \R
\end{equation}
given by
\begin{equation}
	\alpha \cdot \beta := \alpha_Y \cdot \beta_Y
\end{equation}
where $Y$ is any birational model where $\alpha$ is defined.

\subsection{The BPF norm}\label{subsec:BPF-norm}
Let $\cBPF^k (X)$ be the convex cone generated by Cartier classes $\alpha$ such that $\alpha \in \BPF^k (Y)$ for some birational model $Y$ of $X$ and define $\wBPF^k(X)$ to be the weak closure of $\cBPF^k(X)$ inside $\wNk(X)$.
Pick
$\omega \in N^1 (X)$ a big and nef class and define the following norm on $\Vect(\wBPF^k (X))$
\begin{equation}
	\| \alpha \|_{\BPF} = \inf_{\substack{\alpha = \alpha^+ - \alpha^- \\ \alpha^\pm \in \wBPF^k(X)}} \left( \alpha^+ \cdot
	\omega^{d-k} \right) + \left( \alpha^- \cdot \omega^{d-k} \right).
\end{equation}
From \cite{dangSpectralInterpretationsDynamical2021}, we have that $\Vect (\wBPF^k (X))$ is a Banach space equipped with this norm and we define $N^k_{\BPF}(X)$ to be the completion of $\cNk (X)$ in $\Vect (\wBPF^k (X))$ with respect to this norm.
In \cite{dangSpectralInterpretationsDynamical2021}, there is also another Banach space defined with a dual norm using BPF classes and the intersection form.
Namely, for any $\alpha \in \cNk(X)$ define
\begin{equation}
	\| \alpha \|_{\BPF}^{\vee} := \sup_{\substack{\gamma \in \cBPF^{d-k}}(X) \\ \gamma \neq 0} \frac{\left| \alpha
		\cdot \gamma
		\right|}{\omega^k \cdot \gamma}.
\end{equation}
And we write $N^{k, \vee}_{\BPF}$ for the completion of $\cNk (X)$ with respect to $\| \cdot \|^\vee_{\BPF}$.

We define another Banach space for numerical classes in codimension 1.
It is defined using the following norm: let $\omega \in N^1 (X)$ be a big and nef class, then
\begin{equation}
	\forall \alpha \in \cNone (X), \quad \| \alpha \|_{\omega} = \sup_{\substack{\gamma \in
			\cBPF^{d-2} (X) \\ \gamma \cdot \omega^2 = 1}} \left( 2(\alpha \cdot \omega \cdot \gamma)^2 - (\alpha^2 \cdot \gamma)
	\right)^{1/2}.
\end{equation}
We denote by $N^1_\Sigma (X)$ the closure of $\cNone(X)$ in $\wNone (X)$ with respect to this norm.

\begin{prop}
	\label{prop:banach-space}
	Let $X$ be a smooth projective variety over a field $K$ of characteristic zero and let $\tau : \spec K \rightarrow \spec K$ be a finite map.
	The canonical isomorphisms $N^k (Y) \simeq N^k (Y^{(\tau)})$ for any birational model $Y$ of $X$ extend to
	canonical isomorphisms
	\begin{equation}
		\begin{gathered}
		\cNk (X^{(\tau)}) \simeq \cNk (X), \qquad \wNk (X^{(\tau)}) \simeq \wNk (X),\\
		\cBPF^k (X^{(\tau)}) \simeq \cBPF^k(X), \qquad
		\wBPF^k (X^{(\tau)}) \simeq \wBPF^k(X), \qquad
		N^1_\Sigma (X^{(\tau)}) \simeq N^1_\Sigma (X).
		\end{gathered}
	\end{equation}
	with the following relations on the norms:
	\begin{equation}
		\deg(\tau) \cdot \| \cdot \|_{\BPF} = \| \cdot \|_{\BPF^{(\tau)}} \text{ and } \| \cdot \|_{\BPF}^{\vee} = \|
		\cdot \|_{\BPF^{(\tau)}}^\vee, \quad \| \cdot \|_\omega = \| \cdot \|^{(\tau)}_\omega.
	\end{equation}
\end{prop}
\begin{proof}
	The isomorphisms on Cartier and Weil classes follow from Lemma \ref{lemme:numerical-classes-changing-structure}, applied on every smooth birational model.
	The identification of $\cBPF^k$ follows from Lemma \ref{lemme:basepoint-free-classes}, and the identification of $\wBPF^k$ follows by taking weak closures.
	The relation for the BPF norm follows from Lemma \ref{lemme:twist-degree-map}: every top intersection number is multiplied by $\deg(\tau)$.
	For the dual BPF norm, the numerator and denominator are multiplied by the same factor, so the quotient is unchanged.
	Finally, for the $\Sigma$-norm, the condition $\gamma \cdot \omega^2 =1$ is transformed into the same normalized condition after rescaling $\gamma$ by $\deg(\tau)^{-1}$; the expression inside the supremum is therefore unchanged.
\end{proof}

\subsection{Actions of twisted rational maps}\label{subsec:actions-twisted-rational-maps}
Let $X,Y$ be smooth projective varieties of dimension $d$ over a field $K$ of characteristic zero.
Let $f_\tau : X \dashrightarrow Y$ be a twisted rational map, we define a pullback and pushforward operator
\begin{equation}
	f_\tau^* : \cNk (Y) \rightarrow \cNk (X), \quad (f_\tau)_* : \wNk (X) \rightarrow \wNk (Y)
\end{equation}
as follows.
Let $\alpha$ be a Cartier class defined in a birational model $Y'$ and let $X '$ be a birational model of $X$ such that the lift $f_\tau : X' \rightarrow Y'$ is well defined.
We define $f_\tau^* \alpha$ as the Cartier class defined by $(X', f_\tau^* \alpha)$.
If $\beta \in \wNk (X)$, then for any birational model $Y'$ of $Y$ and birational
model $X'$ of $X$ such that $f_\tau : X ' \rightarrow Y'$ is regular, we define
\begin{equation}
	((f_\tau)_* \beta)_{Y'} := ( (f_\tau)_* \beta_{X'}).
	\label{}
\end{equation}
These two linear maps are well defined.
They satisfy the following.
Let $f : X^{(\tau)} \dashrightarrow Y$ be the induced rational map, then by the construction above (which is the same as in \cite{dangSpectralInterpretationsDynamical2021}) we have well defined operators $f^* : \cNk (Y) \rightarrow \cNk (X^{(\tau)})$ and $f_* : \wNk (X^{(\tau)}) \rightarrow \wNk (Y)$, then the operators $f_\tau^*$ and $(f_\tau)_*$ are the compositions of the operators $f^*$ and $f_*$ with the isomorphisms from Proposition \ref{prop:banach-space}.

\begin{prop}
	\label{prop:actions-on-cycles}
		Let $X,Y$ be smooth projective varieties over $K$ and let $f_\tau : X \dashrightarrow Y$ be a twisted rational map,
		then
		\begin{enumerate}
			\item $f_\tau^*$ extends to a continuous linear map $\wNk (Y) \rightarrow \wNk (X)$.
			\item $(f_\tau)_* (\cNk (X)) \subset \cNk (Y)$.
			\item $(f_\tau)_* \left( \cBPF^k (X) \right) \subset \cBPF^k (Y)$.
			\item $(f_\tau)_* \left( \wBPF^k (X) \right) \subset \wBPF^k (Y)$.
		\end{enumerate}
	\end{prop}
\begin{proof}
	Let $f : X^{(\tau)} \dashrightarrow Y$ be the induced rational map.
	By \cite{dangSpectralInterpretationsDynamical2021}, all these statements are true for $f$, so they are true for $f_\tau$ using Proposition \ref{prop:banach-space}.
\end{proof}

\begin{prop}
	\label{prop:action-banach-space}
	Let $X,Y$ be smooth projective varieties of the same dimension $d$ over a field $K$ of characteristic zero and let $f_\tau : X \dashrightarrow
		Y$ be a twisted dominant rational map, then we have linear bounded operators
	\begin{equation}
		f_\tau^* : \sF(Y) \rightarrow \sF(X), \quad (f_\tau)_* : \sF(X) \rightarrow \sF (Y)
	\end{equation}
	where
	\begin{equation}
		\sF =  N^k_{\BPF} (\cdot), N^{k, \vee}_{\BPF}(\cdot), N^1_\Sigma(\cdot).
	\end{equation}
	Fix big and nef classes $\omega_X\in N^1(X)$ and $\omega_Y\in N^1(Y)$, and set
	\[
		\deg_k(f_\tau):=f_\tau^*\omega_Y^k\cdot \omega_X^{d-k}.
	\]
	There exists a constant $C>0$, depending only on $d,k,\omega_X,\omega_Y$, such that
	the following estimates hold on the operators' norms.
	\begin{eqnarray}
		\frac{\deg_k (f_\tau)}{\omega_Y^d} &\leq \| f_\tau^* \|_{\BPF} &\leq C \frac{\deg_k (f_\tau)}{\omega_Y^d} \\
		\frac{\deg_{d-k} (f_\tau)}{\omega_X^d} &\leq \| (f_\tau)_* \|_{\BPF} &\leq C
		\frac{\deg_{d-k}(f_\tau)}{(\omega_X^d)} \\
		\frac{\deg_k (f_\tau)}{\omega_Y^d} &\leq \| f_\tau^* \|_{\BPF}^\vee &\leq C \frac{\deg_k (f_\tau)}{\omega_Y^d} \\
		\frac{\deg_{d-k} (f_\tau)}{\omega_X^d} &\leq \| (f_\tau)_* \|_{\BPF}^\vee &\leq C
		\frac{\deg_{d-k}(f_\tau)}{(\omega_X^d)} \\
		\frac{\deg_1 (f_\tau)}{\omega_X^d} &\leq \| f_\tau^* \|_{\Sigma, \omega} &\leq C \frac{\deg_1 (f_\tau)}{\omega_X^d}.
	\end{eqnarray}
\end{prop}
\begin{proof}
	By \cite{dangSpectralInterpretationsDynamical2021}, all these statements hold for the induced rational map $f : X^{(\tau)} \dashrightarrow Y$ and they follow for $f_\tau$ using Proposition \ref{prop:banach-space}.
	We show the first estimate.
	We have
	\begin{equation}
		\| f^* \|_{\BPF} = \sup_{\alpha \in \cNk (Y), \alpha \neq 0} \frac{\| f^* \alpha \|_{\BPF^{(\tau)}}}{ \|
		\alpha\|_{\BPF}} = \deg (\tau) \| f_\tau^* \|_{\BPF}.
	\end{equation}
	And the result follows from the same estimate with $f$ since
	\begin{equation}
		\deg_k(f) = f^* \omega_Y^k \cdot \omega_{X^{(\tau)}}^{d-k} = \deg (\tau) f_\tau^* \omega_Y^k \cdot
		\omega_X^{d-k}.
	\end{equation}
\end{proof}

\begin{thm}
	\label{thm:dynamical-degrees}
	Let $f_\tau : X \dashrightarrow X$ be a dominant twisted rational map, then for $k = 1, \dots, \dim X$, we have
	\begin{align}
		\lambda_k (f_\tau) & = \rho \left(f_\tau^* | N^k_{\BPF}\right) = \rho \left(f_\tau^* | \Vect (\BPF^k)\right) =
		\rho\left( f_\tau^* | N^{k, \vee}_{\BPF} \right)                                                                           \\
		                   & = \rho \left((f_\tau)_* | N^{d-k}_{\BPF}\right) = \rho \left((f_\tau)_* | \Vect (\BPF^{d-k})\right) =
		\rho\left( (f_\tau)_* | N^{d-k, \vee}_{\BPF} \right)
	\end{align}
	and
	\begin{equation}
		\lambda_1 (f_\tau) = \rho \left( f_\tau^* | N^1_\Sigma \right).
	\end{equation}
	Furthermore, there exists $\theta_k^* \in \wBPF^k (X)$ and $\theta_{k,*} \in \wBPF^{d-k} (X)$ such that
	\begin{equation}
		f_\tau^* \theta^*_k = \lambda_k (f_\tau) \theta^*_k, \quad (f_\tau)_* \theta_{k,*} = \lambda_k (f_\tau)
		\theta_{k,*}.
	\end{equation}

	Finally, if $\lambda_1 (f_\tau)^2 > \lambda_2 (f_\tau)$, then $\theta^*_1 \in N^1_\Sigma (X)$ is unique up to
	multiplication by a positive constant and there exists a continuous linear form $L : N^1_\Sigma (X) \rightarrow
		\R$ such that for any $\alpha \in N^1_\Sigma (X)$
	\begin{equation}
		\frac{1}{\lambda_1 (f_\tau)^n} (f_\tau^n)^* \alpha \xrightarrow[n \rightarrow + \infty]{}
		L(\alpha) \theta^*_1.
		\label{eq:convergence-iterates-divisor}
	\end{equation}
	Furthermore, the linear form $L$ satisfies that $L (\alpha) \geq 0$ whenever $\alpha \geq 0$ and for any big and nef Cartier class $\omega \in \cNone (X)$ we have $\theta^*_1 \cdot \omega^{d-1} > 0$.
\end{thm}
\begin{proof}
	The proof is the same as in \cite{dangSpectralInterpretationsDynamical2021} Theorem 4.11, Theorem 4.12, Theorem 5.2 and Corollary 5.7 using Proposition \ref{prop:action-banach-space}.
	The class $\theta^*_k$ is constructed as follows.
	Let $\omega \in \cNone (X)$ be a big and nef class and define
	\begin{equation}
		\Theta_k^*(t) = \sum_{n \geq 0} t^n (f_\tau^n)^* \omega^k.
	\end{equation}
	Define also $T_k (t) = \Theta_k^* (t) \cdot \omega^{d-k}$.
	The radius of convergence of this power series is $\lambda_k(f_\tau)^{-1}$ and as $t > 0$ increases to $\lambda_k
		(f_\tau)^{-1}$ we can extract a subsequence such that
	\begin{equation}
		\frac{\Theta_k^* (t_n)}{ T_k (t_n)}
	\end{equation}
	converges to an element $\theta_k^* \in \wBPF^k(X)$ that satisfies
	$f_\tau^* \theta^*_k = \lambda_k (f_\tau) \theta^*_k$.
	In particular, we have that $\omega^{d-k} \cdot \theta^*_k = 1$.

	Now, if $\lambda_1(f_\tau)^2 > \lambda_2(f_\tau)$, then the uniqueness of $\theta^*_1$ comes from the same proof as in \cite{dangSpectralInterpretationsDynamical2021}.
	Since we can use any big and nef class $\omega$ for the construction of $\theta^*_1$ from the previous construction we have that $\theta^*_1 \cdot \omega^{d-1} > 0$.
	Therefore taking $\alpha = \omega $ in \eqref{eq:convergence-iterates-divisor} and intersecting in $\omega^{d-1}$ we
	have
	\begin{equation}
		L(\omega) \cdot \theta_1^* \cdot \omega^{d-1} = \lim_n \frac{1}{\lambda_1(f_\tau)^n} \deg_{1, \omega} (f_\tau^n).
	\end{equation}
\end{proof}

\section{Quasi-albanese and differentials}\label{sec:quasi-albanese}
\subsection{Quasi-abelian varieties}\label{subsec:quasi-abelian-varieties}
Let $\overline K$ be an algebraically closed field.
A \emph{quasi-abelian} variety $Q$ is an algebraic group such that there
exists an exact sequence of algebraic groups
\begin{equation}
	0 \rightarrow T \rightarrow Q \rightarrow A \rightarrow 0
\end{equation}
where $T$ is an algebraic torus and $A$ is an abelian variety.
If $U$ is a quasiprojective variety over $\overline K$, then there exists a unique quasi-abelian variety $\QAlb (U)$ with a morphism $\alpha : U \rightarrow \QAlb (U)$ up to isomorphism such that any morphism $f : U \rightarrow Q$ where $Q$ is a quasi-abelian variety factors through $\alpha$.
Notice that if $Q$ is a quasi-abelian variety over $\overline K$ and $\tau : \spec \overline K \rightarrow \spec \overline K$ is an automorphism, then $Q^{(\tau)}$ is also a quasi-abelian variety over $\overline K$.

\begin{lemme}
	\label{lemme:changing-structure-morphism-quasi-albanese}
	For any quasiprojective variety $U$ over $\oK$ and any $\tau : \spec \oK \rightarrow \spec \oK$ a field
	automorphism, we have
	\begin{equation}
		\QAlb(U^{(\tau)}) = \QAlb(U)^{(\tau)}.
	\end{equation}
	In particular, if $f_\tau : U \rightarrow U$ is a twisted morphism, then there exists a twisted morphism $g_\tau : \QAlb(U) \rightarrow \QAlb(U)$ such that $\alpha \circ f_\tau = g_\tau \circ \alpha$.
\end{lemme}
\begin{proof}
	All this follows from the universal property of the Quasi-Albanese variety and Lemma \ref{lemme:change-of-base} (applied with $\tau$ and $\tau^{-1}$).
\end{proof}

\begin{prop}
	\label{prop:absolute-dyn-degree-quasi-abelian}
	Let $Q$ be a quasi-abelian variety of dimension $d$ over a field $K$ of characteristic zero and let $f_\tau: Q \rightarrow Q$ be a twisted endomorphism, then the first dynamical degree $\lambda_1(f_\tau)$ is an algebraic integer of degree $\leq d^2$.
\end{prop}
\begin{proof}
	By base change invariance of dynamical degrees, we may assume that $K$ is algebraically closed.
	Let
	\[
		0\rightarrow T\rightarrow Q\xrightarrow{p} A\rightarrow 0
	\]
	be the extension of an abelian variety by a torus.
	The ordinary morphism $Q^{(\tau)}\to Q$ associated to $f_\tau$ is the composition of a group homomorphism with a translation; write it as $t_q\circ g$, where $g:Q^{(\tau)}\to Q$ is a homomorphism of algebraic groups and $t_q$ is translation by a point of $Q$.
	It follows that $f_\tau$ descends to a twisted endomorphism $\overline f_\tau:A\to A$ and that $p\circ f_\tau=\overline f_\tau\circ p$.
	By the mixed degree formula,
	\[
		\lambda_1(f_\tau)=\max\{\lambda_1(f_\eta),\lambda_1(\overline f_\tau)\},
	\]
	where $f_\eta$ is the induced twisted rational map on the generic fiber of $p$.
	This generic fiber is a torsor under the torus $T_{K(A)}$.
	After a finite extension of $K(A)$ the torsor becomes trivial; by base change invariance, $\lambda_1(f_\eta)$ is unchanged by this extension.
	After such a trivialization, the induced map is a translation composed with a twisted torus endomorphism, so Proposition \ref{prop:absolute-dyn-degrees-algebraic-tori} shows that $\lambda_1(f_\eta)$ is an algebraic integer of degree at most $\dim T$.

	It remains to treat the abelian quotient.
	Translations act trivially on numerical divisor classes, so $\lambda_1(\overline f_\tau)$ is the spectral radius of the action induced by the homomorphism part on $N^1(A)_\R$.
	After spreading out over a finitely generated subfield of $K$ and embedding it into $\C$, we may work over $\C$.
	For an abelian variety of dimension $a=\dim A$, the Neron-Severi group has rank at most $a^2$, and the induced action on it is given by an integral matrix.
	Thus $\lambda_1(\overline f_\tau)$ is an algebraic integer of degree at most $a^2$.
	Since $\dim T\leq d$ and $a^2\leq d^2$, the same bound holds for $\lambda_1(f_\tau)$.
\end{proof}

\subsection{Differentials and logarithmic Kodaira dimension}\label{subsec:log-kodaira-dimension}
We suppose here that $\car \oK = 0$ and that $\oK$ is algebraically closed, if $V$ is a smooth quasiprojective variety over
$K$, the \emph{logarithmic Kodaira dimension} of $V$ is defined as
\begin{equation}
	\overline \kappa (V) = \kappa (\overline V,K_{\oV}+ D_{\oV})
\end{equation}
where $\overline V$ is a smooth projective variety with an open embedding $V \hookrightarrow \overline V$ such that
$\overline V \setminus V = D_{\oV}$ is a simple normal crossing divisor, $K_{\oV}$ is the canonical divisor of
$\overline V$. From
Lemma \ref{lemme:canonical-divisor-unchanged}, we have that for any automorphism $\tau : \spec \oK \rightarrow \spec \oK$,
\begin{equation}
	K_{\oV^{(\tau)}} + D_{\oV^{(\tau)}} \simeq K_{\oV} + D_{\oV}.
\end{equation}

From Proposition 1 of \cite{iitakaLogarithmicKodairaDimension1977}, we get
\begin{prop}
	\label{prop:log-kodaira-iitaka-fibration}
	Let $V$ be a quasiprojective variety over an algebraically closed field $\oK$ of characteristic zero and $f_\tau : V
		\rightarrow V$ a dominant
	twisted endomorphism, then, for every $m \geq 1$, $f_\tau$ induces a linear isomorphism
	\begin{equation}
		f_\tau^*: H^0 \left(\overline V, m(K_{\oV}+D_{\oV})\right) \rightarrow H^0 \left(\overline V^{(\tau)},
		m(K_{\oV^{(\tau)}}+D_{\oV^{(\tau)}})\right).
	\end{equation}
	In particular, let $\kappa = \overline \kappa (V)$. Then there exists an integer $N > 0$, a variety $W \subset
		\P^N$ of dimension $\kappa$, a dominant rational map $q : V \dashrightarrow W$, and a twisted automorphism
	$g_\tau$ of $\P^N$ preserving $W$ such that
	\begin{equation}
		q \circ f_\tau = g_\tau \circ q.
		\end{equation}
	\end{prop}
\noindent\emph{Alternative formulation of Proposition \ref{prop:log-kodaira-iitaka-fibration}.}
Let $V$ be a smooth quasiprojective variety over the algebraically closed field $\oK$ of characteristic zero, and let
$f_\tau:V\to V$ be a dominant twisted endomorphism.
Choose a smooth projective compactification $\overline V$ such that $D=\overline V\setminus V$ is a simple normal crossing divisor.
Then, for every $m\geq 1$, pullback by $f_\tau$ induces a linear isomorphism
\[
	f_\tau^*:H^0\bigl(\overline V,m(K_{\overline V}+D)\bigr)
	\longrightarrow
	H^0\bigl(\overline V^{(\tau)},m(K_{\overline V^{(\tau)}}+D^{(\tau)})\bigr).
\]
In particular, if $\kappa=\overline\kappa(V)\geq 0$ and $m$ is sufficiently divisible, the logarithmic Iitaka map
\[
	q_m:V\dashrightarrow W_m\subset \mathbb P\left(H^0\bigl(\overline V,m(K_{\overline V}+D)\bigr)^\vee\right)
\]
is semiconjugated to a twisted projective linear automorphism $g_\tau$ preserving $W_m$:
\[
	q_m\circ f_\tau=g_\tau\circ q_m.
\]

The proposition also holds in positive characteristic if we assume that $\tau$ and $f_\tau$ are separable.

\subsection{A base change reduction for the Iitaka fibration}\label{subsec:base-change-reduction-iitaka}
\begin{lemme}\label{lemme:semilinear-eigenvector-after-base-change}
	Let $K$ be an algebraically closed field and let $\tau:\spec K\to\spec K$ be a finite map.
	Let $E$ be a finite-dimensional $K$-vector space and let $F_\tau:E\to E$ be a twisted linear automorphism, i.e.
	\[
		F_\tau(av)=\tau^*(a)F_\tau(v)
	\]
	for all $a\in K$ and $v\in E$.
	Then there exist an algebraically closed field extension $\kappa:\spec L\to\spec K$ and an automorphism $\sigma:\spec L\to\spec L$ satisfying
	\[
		\sigma^*\circ \kappa^*=\kappa^*\circ \tau^*
	\]
	such that the base-changed twisted linear automorphism
	\[
		(F_\tau)_{L,\sigma}:E\otimes_K L\longrightarrow E\otimes_K L
	\]
	has a nonzero fixed vector.
\end{lemme}
\begin{proof}
	Since $K$ is algebraically closed, $\tau^*:K\to K$ is an automorphism.
	Choose a basis of $E$ and write
	\[
		F_\tau(v)=A\tau^*(v),\qquad A\in \GL_N(K),
	\]
	where $\tau^*$ is applied coordinatewise.
	Let $L_0=K(x_1,\dots,x_N)$, viewed as an extension of $K$, and write $x=(x_1,\dots,x_N)^t$.
	Define an automorphism $\sigma^*$ of $L_0$ by setting $\sigma^*_{|K}=\tau^*$ and
	\[
		\sigma^*(x)=A^{-1}x.
	\]
	Indeed, the inverse is the homomorphism $\rho^*$ given by $\rho^*_{|K}=(\tau^*)^{-1}$ and
	\[
		\rho^*(x)=(\tau^*)^{-1}(A)x.
	\]
	Then
	\[
		\rho^*(\sigma^*(x))=(\tau^*)^{-1}(A^{-1})(\tau^*)^{-1}(A)x=x
	\]
	and
	\[
		\sigma^*(\rho^*(x))=A A^{-1}x=x,
	\]
	and the same identities are clear on $K$.
	Now choose an algebraic closure $L$ of $L_0$ and extend $\sigma^*$ to an automorphism of $L$.
	We denote by $\kappa: \spec L\to\spec K$ the induced extension of fields, and identify $K$ with its image in $L$.
	For the vector still denoted by $x=\sum_i x_i e_i\in E\otimes_K L$, we have
	\[
		(F_\tau)_{L,\sigma}(x)=A\sigma^*(x)=x.
	\]
	Thus $x$ is a nonzero fixed vector.
\end{proof}

\begin{rmq}
	The automorphism $\sigma$ of the extension field is part of the base change.
	Thus the fixed vector obtained in Lemma \ref{lemme:semilinear-eigenvector-after-base-change} should not be confused with a fixed vector of the original linear map over $K$.
	For instance, suppose that $\tau=\id$ and that $F(v)=Av$ is an ordinary linear automorphism, where $A\in \GL_N(K)$ has no eigenvalue equal to $1$.
	Although $F$ has no nonzero fixed vector over $K$, we may take $L_0=K(x_1,\dots,x_N)$ and define an automorphism $\sigma^*$ of $L_0$ by
	\[
		\sigma^*(x)=A^{-1}x.
	\]
	Then after this twisted base change, the vector $x=\sum_i x_i e_i$ satisfies
	\[
		F_{L_0,\sigma}(x)=A\sigma^*(x)=x.
	\]
	So the fixed vector appears because the ground field has been enlarged together with a nontrivial automorphism $\sigma$.
\end{rmq}

\begin{lemme}\label{lemme:iitaka-generic-fiber-affine}
	Let $X_0$ be a smooth affine variety over an algebraically closed field $K$ of characteristic zero, and let $f_\tau:X_0\to X_0$ be a dominant twisted endomorphism.
	Let $q_m:X_0\dashrightarrow W_m$ be the logarithmic Iitaka map constructed above and let $d_0$ be the dimension of the
  generic fiber, write
	\[
		q_m\circ f_\tau=g_\tau\circ q_m.
	\]
  After an algebraically closed base change $\kappa:\spec L\to\spec K$, together with an automorphism $\sigma:\spec
  L\to\spec L$ satisfying $\sigma^*\circ \kappa^*=\kappa^*\circ\tau^*$, the twisted rational map induced by $f_\tau$ on
  the generic fiber of $q_m$ can be represented by a dominant twisted endomorphism of a smooth affine variety of
  dimension $d_0$ over the function field of the base-changed $W_m$.
\end{lemme}
\begin{proof}
	By Lemma \ref{lemme:semilinear-eigenvector-after-base-change}, applied to the twisted linear automorphism induced by $f_\tau^*$ on the Iitaka linear system, we may make such a base change and find a nonzero section $s$ of this linear system such that, after base change,
	\[
		f_\tau^*s=s.
	\]
	Dynamical degrees are unchanged by this algebraically closed base change, by Proposition \ref{prop:dyn-degrees-base-change}.

	Set
	\[
		U_s:=X_0\cap\{s\neq0\}.
	\]
	Since $X_0$ is affine and $s$ defines an effective Cartier divisor on $X_0$, the open set $U_s$ is affine.
	The rational map $q_m$ is regular on $U_s$ and maps it to the affine chart of $W_m$ where the coordinate corresponding to $s$ is nonzero.
	Moreover $f_\tau^*s=s$ implies
	\[
		f_\tau^{-1}(U_s)=U_s^{(\tau)},
	\]
	so $f_\tau$ restricts to a twisted endomorphism of $U_s$.

	Let $\eta$ be the generic point of this affine chart of $W_m$ and put
	\[
		V_\eta:=U_s\times_{W_m}\spec K(W_m).
	\]
	Then $V_\eta$ is affine.
	Since $U_s$ is smooth and the ground field has characteristic zero, generic smoothness gives that $V_\eta$ is smooth.
	The semiconjugacy $q_m\circ f_\tau=g_\tau\circ q_m$ then restricts to a dominant twisted endomorphism $\tilde f_\eta$ of $V_\eta$ over $K(W_m)$.
\end{proof}

\subsection{Log general type subvarieties of quasi-abelian varieties}\label{subsec:log-general-type-quasi-abelian}
We say that a smooth quasiprojective variety $U$ is of \emph{log general type} if $\overline \kappa(U) = \dim U$.
We state now a result on quasi-abelian varieties that will be useful for the proof of Theorem \ref{bigthm:dyn-degrees-affine-varieties}.

	\begin{thm}[\cite{abramovichSubvarietiesSemiabelianVarieties1994}]\label{thm:quotient-general-type-quasi-abelian}
		Let $Q$ be a quasi-abelian variety and let $V \subset Q$ be closed subvariety.
		Let $G_V = \left\{ x \in Q : x +V = V \right\}$ and let $G_V^0$ be the connected component of the identity of $G_V$; then $V / G_V^0$ is of log general type.
	\end{thm}

	\section{Completions of affine varieties}\label{sec:completions}
Let $X_0$ be a smooth affine variety over an algebraically closed field $K$.
Let $A$ be the ring of regular functions of $X_0$.
A \emph{completion} of $X_0$ is the data of a smooth projective variety $X$ with an open embedding $\iota: X_0 \hookrightarrow X$ such that $\iota(X_0)$ is an open dense subset of $X$.
The boundary $\BD$ is a possibly reducible closed subvariety of pure codimension one.
If $\dim X_0\geq 2$, then $\BD$ is connected: indeed, since $X_0$ is affine, $\BD$ is the support of an effective ample divisor by Goodman's theorem \cite{goodmanAffineOpenSubsets1969}, and the support of an ample divisor on a projective variety of dimension at least two is connected by Hartshorne's connectedness theorem \cite[Chapter III, Corollary 7.9]{hartshorneAlgebraicGeometry1977}.
We denote by $\Div (X)$ the group of divisors of $X$ and by $\Div_\infty (X)$ the subgroup of divisors of $X$ supported on $\partial_X X_0$.
For $\A = \Z, \Q,\R$, we set $\Div (X)_\A := \Div (X) \otimes \A$ and $\DivInf(X)_\A = \DivInf(X) \otimes \A$.
Let $E_1, \cdots, E_m$ be the irreducible components of $\partial_{X} X_0$ (we will call them the \emph{prime divisors at infinity}).
Any element of $\Div_\infty (X)_\A$ is of the form $D = \sum_i a_i(D) E_i$ with $a_i (D) \in \A$.
We will write $\ord_{E_i} (D)$ for $a_i (D)$ of $D$ at $E_i$.
We have also the group $N^1 (X)$ of divisors modulo numerical equivalence and a canonical map $\DivInf(X) \rightarrow N^1 (X)$.
\begin{prop}
	\label{prop:trivial-quasi-albanese-injective-map}
	Let $X$ be a completion of $X_0$. If $\QAlb (X_0) = 0$, then the group homomorphism
	\begin{equation}
		\DivInf (X) \rightarrow N^1 (X)
	\end{equation}
	is injective.
\end{prop}
\begin{proof}
Let $D\in \Div_\infty(X)$ and assume that its numerical class in $N^1(X)$ is zero.
Then the class of $\OO_X(D)$ in $\Pic(X)$ belongs to the kernel of $\Pic(X)\to N^1(X)$, namely to $\Pic^0(X)$.
Since the Albanese variety $\Alb(X)$ is dual to $\Pic^0(X)$ and the Albanese map of $X$ restricts to a morphism
$X_0\to \Alb(X)$, the assumption $\QAlb(X_0)=0$ implies $\Alb(X)=0$.
Thus $\Pic^0(X)=0$, and $\OO_X(D)$ is trivial.
Hence $D=\div(\varphi)$ for some rational function $\varphi\in K(X)^\times$.

Because $D$ is supported on $X\setminus X_0$, the restriction $\varphi_{|X_0}$ has neither zeros nor poles.
Thus $\varphi_{|X_0}\in \OO(X_0)^\times$.
If $\varphi_{|X_0}$ were nonconstant, it would define a nonconstant morphism $X_0\to \G_m$, again contradicting
$\QAlb(X_0)=0$.
Therefore $\varphi_{|X_0}$ is constant, and so $D=\div(\varphi)=0$.
\end{proof}

For a family $(D_j)_{j \in J}$ of elements of $\DivInf (X)$ the coefficients $a_i (D)$ are integers; so, using the
natural order on $\Z$, we
define the supremum $\bigvee_{j \in J} D_j$ and the infimum $\bigwedge_{j \in J} D_j$ by

\begin{equation}
	\bigvee_j D_j = \sum_i \sup(\ord_{E_i}(D_j)) E_i \quad \text { and } \quad \bigwedge_j D_j = \sum_i
	\inf(\ord_{E_i}(D_j)) E_i
\end{equation}

It only exists if each $(\ord_{E_i}(D_j))_{j \in J}$ is bounded respectively from above or from below.
If $\bigwedge_j D_j$ (respectively $\bigvee_j D_j$) is well defined we say that the family $(D_j)$ is \emph{bounded from below (from above)}.
Notice that we only define supremum and infimum for family of divisors with coefficients in $\Z$.

\subsection{Morphisms between completions, Weil, Cartier divisors}\label{SectionCompletions}
Let $X_1, X_2$ be two completions of $X_0$.
We denote their embeddings by $\iota_1$ and $\iota_2$.
There exists a unique birational map $\pi: X_1 \dashrightarrow X_2$ such that the diagram
\begin{equation}
	\label{DiagMorphismCompletions}
	\begin{tikzcd}
		X_1 \ar[r, dashed, "{\pi}"] & X_2 \\ X_0 \ar[u, "{\iota_1}", hook] \ar[r, equal, "{\id}"] & X_0 \ar[u, "{\iota_2}", hook]
	\end{tikzcd}
\end{equation}
commutes.
If $\pi$ is a morphism, we call it a \emph{morphism of completions}.
In that case we say that $X_1$ is \emph{above} $X_2$.
Since $\pi$ is an isomorphism over $X_0$, it is the blowup of an ideal sheaf on $X_2$ supported on
$X_2\setminus X_0$.
Any two completions of $X_0$ are dominated by a third one.

If $\pi : X_1 \rightarrow X_2$ is a morphism of completions we have the push-forward and pull-back operators.
They define
group homomorphisms
\begin{equation}
	\pi_* : \Div_\infty (X_1)_\A \twoheadrightarrow \Div_\infty (X_2)_\A \quad \text{and} \quad \pi^* :
	\Div_\infty (X_2)_\A \hookrightarrow \Div_\infty (X_1)_\A.
	\label{EqHomomorphisAtInfinity}
\end{equation}

Let $X$ be a completion of $X_0$ and $P \in K[X_0]$ be a regular function, then $P$ induces a rational function over $X$ and we write $\div_X (P)$ for its associated Weil divisor.
In particular, if $\pi: Y \rightarrow X$ is a morphism of completions above $X_0$, then
\begin{equation}
	\div_Y (P) = \pi^* \div_X (P).
\end{equation}
We will write $\div_{\infty, X} (P) \in \DivInf (X)$ the divisor on $X$ supported at infinity such that \[ \div_X (P) = D + \div_{\infty, X} (P) \] where $D$ is an effective divisor and no components of its support is in $\BD$.

The system of completions of $X_0$ is a projective system.
Similarly as in \S\ref{sec:banach-spaces} we define the
space of Cartier and Weil divisors at infinity by
\begin{equation}
	\Cinf_\A = \varinjlim_{X} \DivInf(X)_\A, \text{ and } \Winf_\A = \varprojlim_X \DivInf(X)_\A
\end{equation}
where the limits are over all completions of $X_0$ (we do not allow blow up above $X_0$ here).
We have a natural inclusion
\begin{equation}
	\Cinf_\A \hookrightarrow \Winf_\A.
\end{equation}

An element $D$ of $\Winf_\A$ with $\A = \Z, \Q, \R$ is called \emph{effective} (denoted by $D \geq 0$) if its incarnation in every completion $X$ is effective; if $D$ belongs to $\Cinf_\R$ this is equivalent to $D_X \geq 0$ for one completion $X$ where $D$ is defined.
If $D_1, D_2 \in \Winf_\A$, we will write $W_1 \geq W_2$ for $W_1 - W_2 \geq 0$.
\subsection{Relations with Banach spaces}\label{subsec:relations-with-banach-spaces}
Let $X_0$ be an affine variety, we define the space of Cartier and Weil classes $\cNk(X_0), \wNk(X_0)$ of $X_0$ as $\cNk(X), \wNk(X)$ for any completion $X$ of $X_0$.
This does not depend on the choice of the completion $X$.
We define all the other Banach spaces similarly using any completion of $X_0$.
If $\QAlb(X_0) = 0$, then by Proposition \ref{prop:intersection-form-non-degenerate-at-infinity} we have a natural embedding
\begin{equation}
	\Cinf_\R \hookrightarrow \cNone (X_0).
\end{equation}
There is also a well defined embedding $\Winf_\R \hookrightarrow \wNone(X_0)$ but it is more tricky to define and not needed for our purposes.

\begin{prop}
	\label{prop:pullback-preserves-cinf}
	Let $f_\tau : X_0 \rightarrow X_0$ be a dominant endomorphism of an affine variety, then $f_\tau^*$ preserves $\Cinf$.
\end{prop}
\begin{proof}
	We already have that $f_\tau^* : \cNone (X_0) \rightarrow \cNone(X_0)$.
	Let $D \in \Cinf$ and let $X$ be a completion of $X_0$ where $D$ is defined.
	Let $Y$ be another completion of $X_0$ such that $f_\tau$ extends to a twisted regular map $f_\tau : Y \rightarrow X$, then we have that $f_\tau^{-1} (X \setminus X_0) \subset Y \setminus X_0$ because $f_\tau : X_0 \rightarrow X_0$ is a twisted endomorphism so that $f_\tau^* D \in \DivInf(Y) \subset \Cinf$.
\end{proof}
Notice here that this proposition holds without any properness assumption.

\subsection{Divisorial valuations}\label{subsec:divisorial-valuations}
Let $\sD_\infty (X_0)$ be the set of prime divisors supported at infinity over every completion $X$ of $X_0$ modulo the following equivalence relation: if $X, X'$ are two completions of $X_0$ with prime divisors $E, E'$, then $E$ is equivalent to $E'$ if and only if the birational map of completions $X \dashrightarrow X'$ induces a birational map from $E$ to $E'$.
For any $E \in \sD_\infty(X_0)$, we can define the map $\ord_E : \Winf_\R \rightarrow \R $ defined by
\begin{equation}
	\ord_E (W) := \ord_E (W_X)
\end{equation}
for any $X$ such that $E$ is a prime divisor at infinity in $X$.
It is clear that this does not depend on the choice of $X$ or $E$ and we have the following interpretation of $\ord_E (W_X)$ if $W_X$ is Cartier: the local ring at the generic point $\eta_E$ of $E$ is a discrete valuation ring with a valuation that we also write $\ord_E$ and if $h$ is a local equation of $W_X$ at $\eta_E$, then $\ord_E (W_X) = \ord_E (h)$.
This does not depend on the choice of the local equation $h$.

\subsection{Supremum and infimum of divisors}\label{SubSecSupremumAndInfimum}

Let $(D_i)_{i \in I}$ be a family of elements of $\Winf$ such that for all completions $X$, the
family $(D_{i,X})$ is bounded from below, we define $\bigwedge_{i \in I}
	D_i$ with its incarnation in $X$ being
\begin{equation}
	\left(\bigwedge D_i\right)_{X} = \bigwedge_i D_{i, X}.
\end{equation}
We have an analogous definition for $\bigvee_i D_i$ when each $(D_{i,X})$ is bounded from above.

\begin{lemme}
	\label{LemmeMinOfCartierIsCartier}
	If $D,D' \in \Cinf$, then $D \wedge D', D \vee D' \in \Cinf$.
\end{lemme}

\begin{proof}
	It suffices to show that $D \wedge D' \in \Cinf$ because $D \vee D' = - (-D \wedge -D')$.
	So take $D, D' \in \Cinf$, we have to show that $D \wedge D'$ belongs to $\Cinf$.

	Now, it suffices to show this for $D, D'$ effective, indeed let $X$ be a completion such that $D$ and $D'$ are defined over $X$.
	Then, there exists $D_2 \in \Div_\infty (X)$ such that $D - D_2$ and $D' - D_2$ are effective.
	Indeed, take $D_2$ as the Cartier class determined by $D \wedge D'$ in $X$, Then
	\begin{equation}
		D \wedge D' = (D - D_2) \wedge (D' - D_2) + D_2.
	\end{equation}

	Therefore, suppose $D, D'$ are effective.
	Then $\aa = \OO_{X} (-D) + \OO_{X} (- D')$ is a coherent sheaf of ideals such that $\aa_{|X_0} = \OO_{X_0}$.
	Let $X'\to X$ be the blowup of $X$ along $\aa$.
	Since $\aa_{|X_0}$ is trivial, this blowup is an isomorphism over $X_0$.
	Choose a smooth completion $Y$ of $X_0$ dominating $X'$; after replacing $X'$ by such a resolution, we obtain a morphism of completions $\pi:Y\to X$.
	Then, $\bb := \pi^* \aa \cdot \OO_{Y}$ is an invertible sheaf over $Y$ trivial over $X_0$, so there exists a divisor $D_{Y} \in \DivInf(Y)$ such that $\bb = \OO_{Y} (- D_{Y})$.

	\begin{claim}
		\label{ClaimPullBackOfIdealSheafIsMinimumCartierClass}
		The Cartier class in $\Cinf$ induced by $D_{Y}$ is $D \wedge D'$.
	\end{claim}
	\begin{proof}
		Let $E \in \sD (X_0)$ be a prime divisor at infinity.
		We can suppose that $E$ is contained in some completion $Z$ above $Y$.
		Write $\tau : Z \rightarrow Y$ for the morphism of completions and let $h, h '$ be local equations of $D$ and $D'$ respectively at $\pi (\tau (\eta_E))$ where $\eta_E$ is the generic point of $E$ in $Z$.
		If $g$ is a local equation of $D_Y$ at $\tau (\eta_E)$, then in $\OO_{Z, \eta_E}$ we have that the ideal generated by $\tau^* g$ is exactly the ideal generated by $\pi^* \tau^* h, \pi^* \tau^* h'$.
		Since the ideals of $\OO_{Z, \eta_E}$ are of the
		form $(z^k)$ where $\ord_E (z) = 1$ we have that
		$\ord_E (\tau^* g) = \min \left(\ord_E(\tau^* \pi^* h), \ord_E(\tau^* \pi^* h')\right)$
		and therefore
		\begin{equation}
			\ord_E (D_Y) = \min (\ord_E (D), \ord_E (D'))
		\end{equation}
		and this shows that $D_Y = D \wedge D'$.
		\end{proof}
	\end{proof}

The proof actually gives the following corollary.
If $X$ is a variety and $Y \subset X$ is a closed subvariety, by blowing up $Y$ in $X$ we mean blowing up the defining ideal sheaf of $Y$.

\begin{cor}[Corollary of the proof]\label{cor:blowup-intersection}
	Let $X$ be a completion of $X_0$ and $Y, Y'$ be two closed irreducible subvarieties of $X$.
	If $D, D'$ are the exceptional divisors obtained by blowing up respectively $Y$ and $Y'$, then the exceptional divisor $D_\cap$ obtained by blowing up $Y \cap Y'$ is exactly $D \wedge D'$.
\end{cor}
\begin{proof}
	Let $\cI$ and $\cI '$ be the ideal sheafs in $X$ defining $Y$ and $Y'$, blowing up $Y \cap Y'$ means that we blow up the ideal sheaf $\cI + \cI '$.
	Let $E \in \sD (X_0)$ and $\pi: Z \rightarrow X$ be a completion that contains $E$ and dominates the blowup of $X$ with respect to $\cI, \cI '$ and $\cI + \cI '$.
	Let $f_1 ,\dots, f_r$ be local generators of $\cI$ at $\pi (\eta_E)$ and $g_1 ,\dots, g_s$ be local generators of $\cI '$ at $\pi (\eta_E)$.
	Then, $(f_1, \dots, f_r, g_1, \dots, g_s)$ is a set of local generators of $\cI + \cI '$ at $\pi (\eta_E)$ and by a similar argument to the proof of the claim, we have $\ord_E (D) = \min_i \ord_E (f_i)$, $\ord_E (D') = \min_j \ord_E (g_j)$ and $\ord_E (D_\cap) = \min \left(\min_i\ord_E (f_i), \min_j\ord_E (g_j)\right) = \min (\ord_E (D), \ord_E (D'))$.
\end{proof}

\section{Valuations and linear forms}\label{sec:valuations-linear-forms}
Let $X_0$ be a an affine variety with ring of regular functions $A$.
A valuation over $A$ is a map $v : A \rightarrow \R
	\cup \left\{ + \infty \right\}$ such that
\begin{enumerate}
	\item $\forall P,Q \in A, \quad v(PQ) = v(P) + v (Q)$.
	\item $\forall P,Q \in A, \quad v(P+Q) \geq \min (v(P), v (Q))$.
	\item $v (0) = +\infty$.
	\item $v_{|k^*} = 0$.
\end{enumerate}

\begin{thm}[Abhyankar's inequality]\label{thm:abyankhar}
	Let $K$ be a field over a subfield $k$ and $v : K \rightarrow \R$ be a valuation over $K$.
	If $\Gamma_v$ is the value
	group of $v$, then
	\begin{equation}
		\dim_\Q \Gamma_v \otimes \Q \leq \trdeg K / k.
	\end{equation}
\end{thm}

For every completion $X$ of $X_0$, a valuation $v$ admits a \emph{center} $c_X (v)$ which is a point of the scheme $X$.
We say that a valuation is \emph{centered at infinity} if there exists a completion $X$ such that $c_X (v) \not \in X_0$.
As in \cite{abboudDynamicsEndomorphismsAffine2023}, every valuation $v$ over $X_0$ centered at infinity induces a
linear form
\begin{equation}
	L_v : \Cinf \rightarrow \R
\end{equation}
such that
\begin{enumerate}[label=(L\arabic*)]
	\item \label{condition:L1}
	      If $D \geq 0, L_v (D) \geq 0$.
	\item \label{condition:L2} $\forall D, D' \in \Cinf, \quad L_v (D \wedge D') = \min \left( L_v (D), L_v (D') \right)$.
\end{enumerate}
The definition of $L_v$ is as follows.
Let $D \in \Cinf$ and let $X$ be a completion where $D$ is defined.
Let
$\phi$ be a local equation of $D$ at $c_X (v)$, then we set
\begin{equation}
	L_v (D) := v (\phi).
\end{equation}
In particular, $L_v = 0 \Leftrightarrow c_X (v) \in X_0$ for some completion $X$ of $X_0$.

Conversely, a linear form $L : \Cinf \rightarrow \R$ satisfying these two conditions defines a valuation $v_L$ centered
at infinity defined by
\begin{equation}
	\forall P \in A, \quad v_L (P) = \sup_X L (\div_{X, \infty} (P))
\end{equation}
where $\div_{X, \infty} (P)$ is the part at infinity of the divisor of the rational function of $P$ in $X$.
More precisely, let $\cS_\infty (X_0) \subset \Winf$ be the space of Weil divisors which can be written as a supremum of Cartier divisors at infinity, this defines a semigroup.
If $L$ is a linear form on $\Cinf$ satisfying Conditions \ref{condition:L1} and \ref{condition:L2}, then we can extend $L : \cS_\infty (X_0) \rightarrow \R \cup \left\{ + \infty \right\}$ by setting
\begin{equation}
\text{if } W = \sup_{i \in I} D_i, \quad L(W) := \sup_i L(D_i)
\end{equation}
This does not depend on the definition of $W$ as a supremum and $L : \cS_\infty (X_0) \rightarrow \R \cup \left\{ +\infty \right\}$ is a morphism of semigroups.
Furthermore, we have that if $W_1, W_2 \in \cS_\infty (X_0)$ then $W_1 \wedge W_2 \in \cS_\infty (X_0)$ as
\begin{equation}
  W_1 = \sup_i \alpha_i, \quad W_2 =\sup_j \beta_j, \quad W_1 \wedge W_2 = \sup_{i,j} \alpha_i \wedge \beta_j.
\end{equation}
And this yields 
\begin{equation}
  L (W_1 \wedge W_2) = \min (L(W_1), L(W_2)).
\end{equation}
Now, since for any morphism of completions $ \pi : Y \rightarrow X$ we have $\div_{\infty, Y} (P) \geq \div_{\infty, X} (P)$ and $\pi_* \div_{\infty, Y} (P) = \div_{\infty, X} (P)$. The Weil divisor $\div_\infty(P) := \left( \div_{\infty, X} (P) \right)_X$ belongs to $\cS_\infty (X_0)$ as 
\begin{equation}
  \div_{\infty} (P) = \sup_X \div_{\infty, X} (P).
\end{equation}
And we set 
\begin{equation}
  \forall P \in \OO(X_0), \quad v_L (P) = L (\div_\infty (P)).
\end{equation}
	Since $\div_{\infty} (PQ) = \div_{\infty} (P) + \div_{\infty} (Q)$ and $\div_{\infty} (P+Q) \geq \div_{\infty} (P) + \div_{\infty} (Q)$ we get that $v_L $ is a valuation over $\OO(X_0)$.
	
	  It is not clear a priori that $v_L$ is centered at infinity but we show in \S \ref{subsec:proof-bijection} that the two processes $v \mapsto L_v$ and $L \mapsto v_L$ are bijective.

\subsection{Dual curves associated to a divisorial valuation}\label{subsec:dual-curves-divisorial-valuations}
Suppose $\QAlb (X_0) = 0$.
Then we have seen that $\Cinf \hookrightarrow \cNone (\sX)$ for any completion $X$ of $X_0$.
Let $v$ be a divisorial valuation over $X_0$, it defines a linear form $L_v$ over $\Cinf$.
By duality this defines a class of curve $\gamma_v \in \wNone(X_0)$ which is unique up to $\Cinf^\perp$.

\begin{dfn}
	\label{dfn:dual-curve}
	If $v$ is a divisorial valuation over $X_0$, we denote by $\gamma_v$ any \emph{Cartier} class $\gamma_v \in \cNone(X_0)$ such that $\gamma_v \cdot$ induces $L_v$ over $\Cinf$.
\end{dfn}

Here is an example of how to construct $\gamma_v$.
Let $Y$ be a completion of $X_0$ such that $v = \lambda \ord_E$ and $E$ is a prime divisor at infinity in $Y$.
By the Hodge-Riemann formula, if $H$ is an ample Cartier divisor over $Y$,
then the symmetric bilinear form
\begin{equation}
	(\alpha, \beta) \in N^1 (Y)_\R^2 \mapsto H^{d-2} \cdot \alpha \cdot \beta
\end{equation}
is non-degenerate of signature $(1, \rho(Y) - 1)$.

\begin{prop}
	\label{prop:intersection-form-non-degenerate-at-infinity}
	Suppose $Y$ admits an ample divisor supported at infinity, then the restriction of the bilinear form
	\begin{equation}
		(\alpha, \beta) \in \DivInf(Y) \times \DivInf (Y) \mapsto H^{d-2} \cdot \alpha \cdot \beta
	\end{equation}
	is also non-degenerate of signature $(1, r-1)$ where $r$ is the number of irreducible components of $\partial_Y X_0$.
\end{prop}
\begin{proof}
	Let $H_0$ be an ample divisor supported at infinity in $Y$, then we have that $H^{d-2} \cdot H_0^2 > 0$.
	Now if $D \in \DivInf(Y)$ is orthogonal to $\DivInf(Y)$, then
	\[
		H^{d-2}\cdot H_0\cdot D=0
		\quad\text{and}\quad
		H^{d-2}\cdot D^2=0.
	\]
	Since $\DivInf(Y)\hookrightarrow N^1(Y)$, a nonzero such $D$ would be a nonzero class orthogonal to the positive class $H_0$ for the Hodge-Riemann form.
	Thus the Hodge-Riemann theorem gives $H^{d-2}\cdot D^2<0$, a contradiction.
	Hence the restriction of the form to $\DivInf(Y)$ is non-degenerate.
\end{proof}

Now, write $E_1, \dots, E_r$ for the irreducible components of $Y$ at infinity.
This is a basis of $\DivInf(Y)_\R$ and we write $\hat E_1, \dots, \hat E_r$ the dual basis for the bilinear product induced by an ample divisor $H$.
Then, $\gamma_{\ord_{E_i}} = H^{d-2} \cdot \hat E_i$ are Cartier classes of curves defined in $Y$ that induce $L_{\ord_{E_i}}$ over $\Cinf$.
Indeed, if $D \in \Winf$, then
	\begin{equation}
		L_{\ord_{E_i}}(D) = L_{\ord_{E_i}} (D_Y) = \gamma_{\ord_{E_i}} \cdot D_Y = \gamma_{\ord_{E_i}} \cdot D.
	\end{equation}

We now define the pushforward on valuations centered at infinity.
Let $f_\tau:X_0\to X_0$ be a dominant twisted endomorphism and let $v\in\Vinf$.
We define the pushforward valuation $(f_\tau)_*v$ on $K[X_0]$ by
\[
	((f_\tau)_*v)(P)=v(f_\tau^*P).
\]
If its center lies in $X_0$, then it induces the zero linear form on $\Cinf$, since every divisor at infinity is disjoint from its center.
In this case, when we work in the cone of valuations centered at infinity with a zero element added, we identify $(f_\tau)_*v$ with $0$.
For a linear form $L$ on $\Cinf$, we define
\[
	((f_\tau)_*L)(D):=L(f_\tau^*D),\qquad D\in\Cinf.
\]

\begin{prop}
	\label{prop:pushforward-dual-curves}
	Let $v$ be a divisorial valuation and let $\gamma_v$ be a dual curve associated to $v$, then for any dominant
	twisted endomorphism $f_\tau$ of $X_0$ we have in $\cNone(X_0)$,
	\begin{equation}
		(f_\tau)_* \gamma_v = \gamma_{(f_\tau)_*v} \mod \Cinf^\perp.
	\end{equation}
\end{prop}
\begin{proof}
	This follows from the corresponding statement for the morphism associated to $f_\tau:X_0\to X_0$ after twisting the source.
	Indeed, $\Cinf$ is invariant under $f_\tau^*$ and, for every $D\in \Cinf$,
	\begin{equation}
		L_{(f_\tau)_*v}(D) = L_v (f_\tau^* D).
	\end{equation}
	The defining duality between $\gamma_v$ and $L_v$ then gives the asserted equality modulo $\Cinf^\perp$.
\end{proof}

\subsection{Proof of the bijection}\label{subsec:proof-bijection}
\begin{thm}
	\label{thm:bijection}
	For every valuation $v \in \Vinf$, we have
	\begin{equation}
		v_{L_v} = v.
	\end{equation}
  And for every linear form over $\Cinf$ that satisfies Conditions \ref{condition:L1} and \ref{condition:L2} we have
	\begin{equation}
		L_{v_L} = L.
	\end{equation}
\end{thm}
\begin{proof}[Proof of the first assertion]
	We prove that $v_{L_v}=v$.
	Let $P\in k[X_0]$.
	If $P=0$ there is nothing to prove, so we assume $P\neq 0$.
	Choose a completion $X$ on which the rational map $P:X\dashrightarrow \mathbb P^1$ is regular, and write
	\[
		\div_X(P)=D_X+\div_{X,\infty}(P),
	\]
	where $D_X$ is the effective divisor of zeros of $P$ on $X_0$ and $\div_{X,\infty}(P)$ is supported at infinity.
	Since $P$ is a local equation of $\div_X(P)$ at the center of $v$, we have
	\[
		v(P)\geq L_v(\div_{X,\infty}(P)).
	\]
	Taking the supremum over $X$ gives $v(P)\geq v_{L_v}(P)$.

	It remains to prove the reverse inequality.
	If for some completion $X$ the center $c_X(v)$ does not belong to $\Supp D_X$, then near $c_X(v)$ the divisor of $P$ is exactly its part at infinity.
	Hence
	\[
		v(P)=L_v(\div_{X,\infty}(P))\leq v_{L_v}(P),
	\]
	and we are done.

	Assume now that $c_X(v)\in\Supp D_X$ for every completion $X$.
	After blowing up over the boundary, we may suppose near $c_X(v)$ that $\Supp\div_X(P)$ has simple normal crossings, that there is a unique irreducible component $Y$ of $D_X$ through $c_X(v)$, and that the boundary components through $c_X(v)$ are $E_1,\dots,E_r$ with local equations $x_1,\dots,x_r$.
	Let $z$ be a local equation of $Y$.
	If $v(z)<+\infty$, set
	\[
		\Delta=\min\{v(z),v(x_1),\dots,v(x_r)\}>0.
	\]
	Blow up the smooth center $Y\cap E_1\cap\cdots\cap E_r$.
	If $Y'$ is the strict transform of $Y$, then at the new center of $v$ a local equation $z'$ of $Y'$ satisfies
	\[
		v(z')=v(z)-\Delta.
	\]
	Repeating the construction, after finitely many steps the center of $v$ is no longer contained in the strict transform of $Y$, contradicting the present assumption.
	Therefore this case cannot occur.

	Thus necessarily $v(z)=+\infty$.
	With $\Delta=\min\{v(x_1),\dots,v(x_r)\}>0$, the same sequence of blowups gives completions $X_m$ such that
	\[
		L_v(\div_{X_m,\infty}(P))
		=
		L_v(\div_{X,\infty}(P))+m\Delta.
	\]
	Letting $m\to\infty$ gives $v_{L_v}(P)=+\infty=v(P)$.
	This proves $v_{L_v}=v$.
\end{proof}

Now it remains to show that for any linear form $L : \Cinf \rightarrow \R$ satisfying the two conditions we have $L_{v_L} = L$.
To do so we need to define the \emph{center} of such a linear form.
\subsubsection{The center of $L$}\label{subsubsec:center-of-L}
For a linear form satisfying the two conditions, we can define its center using Corollary \ref{cor:blowup-intersection}.
\begin{prop}
	For every $X$, there exists a minimal closed subvariety $Y \subset X$ such that $Y \subset \BD$ and if $Z \rightarrow X$ is the blowup of $X$ with respect to $Y$ with exceptional divisor $\tilde E$, then $L (\tilde E) > 0$.
	The generic point of $Y$ is called the \emph{center} of $L$ over $X$ and we denote it by $c_X (L)$.
\end{prop}
\begin{proof}
	Let $Y_1, Y_2$ be two subvarieties of $X$ satisfying this property and let $\tilde E_1, \tilde E_2$ be the exceptional divisor obtained when blowing up respectively $Y_1$ and $Y_2$.
	Then, let $Y = Y_1 \cap Y_2$, if $\tilde E$ is the exceptional divisor obtained when blowing up $Y$, then by Corollary \ref{cor:blowup-intersection}, $\tilde E = \min (\tilde E_1, \tilde E_2)$ and $L (\tilde E) > 0$.
	In particular, $Y_1$ and $Y_2$ must intersect.

	Also there must exist an irreducible component $E \subset \BD$ such that $L(E)>0$ otherwise for any completion $\pi: Z \rightarrow X$ above $X$ and any effective divisor $D_Z \in \DivInf (Z)$, there exists a divisor $D_X \in \DivInf(X)$ such that $D_X \geq D_Z$ and we would get $L (D_Z) = 0$ and therefore $L = 0$.

		Thus the collection of such subvarieties is non-empty and is stable under pairwise intersections.
		By noetherianity, it has a minimal element.
		So the center of $L$ is defined as this minimal subvariety $Y$ such that blowing up $Y$ yields an exceptional divisor $\tilde E$ such that $L(\tilde E) > 0$.
	\end{proof}

We now show that $L_{v_L} = L$.
We say that $L$ is \emph{divisorial} if there exists $X$ such that the center of $L$ over $X$ is the generic point of an irreducible component $E$ of $\BD$.
In that case, $v_L = L(E) \ord_E$ and it not hard to see that $L_{v_{L}} = L$.

Suppose that $L$ is not divisorial.
Let $X$ be a completion of $X_0$ with very ample boundary, whose existence follows from Goodman's theorem, and let $H$ be a very ample divisor over $X$ such that $\Supp H = \BD$.
Pick $P$ general in $\Gamma (X, H)$, we have that $\div_X (P)$ is of the form
\begin{equation}
	\div_X (P) = Z_P - H
\end{equation}
where $Z_P$ is an effective divisor supported over $X_0$ and the center of $L$ in $X$ does not belong to $\Supp Z_P$.
Then, $1/P$ is a local equation of $H$ at $c_X (L)$ and
\begin{equation}
	v_L (P) = \sup_Z L (\div_{Z, \infty} (P)).
\end{equation}

Now, for every morphism of completions $\pi : Z \rightarrow X$, we have
\begin{equation}
	\div_Z (P) = \pi^* Z_P - \pi^* H = \pi ' Z_P + D_Z - \pi^* H.
\end{equation}
And $\div_{Z, \infty} (P) = D_Z - \pi^* H$ but by definition of $c_X (L)$ we must have $L (D_Z) = 0$ and therefore
\begin{equation}
	L (\div_{Z, \infty} (P)) = -L(H).
\end{equation}
Therefore, we get
\begin{equation}
	L_{v_L} (H) = v_L (1/P) = -v_L(P)= L (H),
\end{equation}
	where we use the extension of $v_L$ to the function field.
	Now, if $D \in \DivInf (X)$ is any divisor, then there exists an integer $m > 0$ such that $D + mH$ is effective and
	very ample and we have
\begin{equation}
	L_{v_L} (D) = L_{v_L} (D + mH) - L_{v_L} ( mH) = L (D +mH) - L (mH) = L(D).
\end{equation}
	The result is shown.

\begin{cor}
	\label{cor:pushforward-linear-form}
	let $f_\tau :X_0 \rightarrow X_0$ be a twisted dominant endomorphism, then
	\begin{equation}
		v_{(f_\tau)_* L} = (f_\tau)_* v_L.
	\end{equation}
\end{cor}
\begin{proof}
	By the definitions above, for any valuation $v\in\Vinf$ we have
	\begin{equation}
		L_{(f_\tau)_* v}= (f_\tau)_*L_v,
	\end{equation}
	where both sides are understood to be zero if $(f_\tau)_*v$ has center in $X_0$.
	Applying this to $v=v_L$ and using $L_{v_L}=L$ gives
	\begin{equation}
		L_{(f_\tau)_* v_L}= (f_\tau)_*L.
	\end{equation}
	The bijection between valuations centered at infinity, with the zero element allowed, and linear forms satisfying the two conditions gives the result.
\end{proof}

\begin{thm}
	\label{thm:eigenvaluations}
	Let $X_0$ be a smooth affine variety over a field of characteristic zero with $\QAlb(X_0) = 0$.
	Let $f_\tau : X_0 \rightarrow X_0$ be a twisted dominant endomorphism such that $\lambda_1 (f_\tau)^2 > \lambda_2(f_\tau)$, then there exists a valuation $v_*$ centered at infinity such that $(f_\tau)_* v_* = \lambda_1 (f_\tau) v_*$.
\end{thm}
\begin{proof}
	By Theorem \ref{thm:dynamical-degrees}, there exists a unique $\theta^* \in N^1_\Sigma (X_0) \cap \wBPF^1(X_0)$ and a
	continuous linear form $L$ on
	$N^1_\Sigma(X_0)$ such that for any $\alpha \in N^1_\Sigma (X_0)$
	\begin{equation}
		\frac{1}{\lambda_1(f_\tau)^n} (f_\tau^n)^* \alpha \xrightarrow[n \rightarrow +\infty]{}
		L(\alpha) \theta^* \label{eq:iteration}
	\end{equation}
	such that $\alpha \geq 0 \Rightarrow L(\alpha) \geq 0$ and for any $\omega \in \cNone(X)$ big and nef, $\theta^* \cdot \omega^{d-1} > 0$.
	In particular, $L$ induces a continuous linear form on $\Cinf$ since $\Cinf \hookrightarrow \cNone(X_0)$.
	We show that $L$ satisfies Conditions \ref{condition:L1} and \ref{condition:L2}.
	It is clear that $L$ satisfies Condition \ref{condition:L1}.
	We show that $L$ satisfies Condition \ref{condition:L2}.
	By Goodman's theorem we have that there exists $X$ a completion of $X_0$ such that $X$ admits an ample effective divisor $H$ with $\Supp H = X \setminus X_0$.
	There must exist a divisor at infinity $E$ in $X$ such that $\gamma_{\ord_E} \cdot \theta^* > 0$.
	Indeed, we have that $\theta^* \cdot H^{d-1} > 0$.
  Let $\gamma_{\ord_{E_i}} = H^{d-2}\cdot \hat E_i$ be the classes of curves associated to the components at infinity in $X$ as in \S \ref{subsec:dual-curves-divisorial-valuations}. Since $\hat E_i$ is the dual basis of $E_i$ for the intersection product induced by $H^{d-2}$ we have 
  \begin{equation}
  \theta^* \cdot H^{d-1} = \theta^*_X \cdot H^{d-1} = \sum_i (H^{d-1}\cdot E_i) H^{d-2} \cdot \hat E_i \cdot \theta_X^* = \sum_i \left( H^{d-1} \cdot E_i \right) \gamma_{\ord_{E_i}} \cdot \theta^*.
\end{equation}
And since $H^{d-1} \cdot E_i > 0$ we have that there must exist $i$ such that $\gamma_{\ord_{E_i}} \cdot \theta^* > 0$.

	Now, let $D \in \Cinf$, recall that $f_\tau^*$ preserves $\Cinf$ even if $f_\tau$ is not proper by Proposition \ref{prop:pullback-preserves-cinf}.
	Intersecting
	\eqref{eq:iteration} with $\gamma_{\ord_E}$ (which is possible because
	$\gamma_{\ord_E} \in N_1(X) \subset \cNN_1(X_0)$ and the intersection product is continuous) we have that
	\begin{equation}
		\gamma_{\ord_E} \cdot\left(\frac{1}{\lambda_1(f_\tau)^n} (f_\tau^n)^* D\right) =  L_{\ord_E} \left(\frac{1}{\lambda_1(f_\tau)^n}
			(f_\tau^n)^* D\right) =
		L_{\frac{1}{\lambda_1(f_\tau)^n}(f_\tau^n)_* \ord_E} (D) \rightarrow
		L(D) \gamma_{\ord_E}\cdot \theta^*.
		\label{ }
	\end{equation}
	To show that $L$ satisfies Condition \ref{condition:L2} we have to prove
	\begin{equation}
		\forall D, D' \in \Cinf, \quad L(D \wedge D') = \min (L(D), L(D')).
	\end{equation}
	Write $\gamma_{\ord_E} \cdot \theta^* =: c > 0.
	$ We have
	\begin{align}
		c \cdot L(D \wedge D') & = \lim_n L_{\ord_E}\left(\frac{1}{\lambda_1(f_\tau)^n}  (f_\tau^n)^* (D \wedge D')\right)                            \\
		                       & = \lim_n L_{\frac{1}{\lambda_1(f_\tau)^n}(f_\tau^n)_* \ord_E} (D \wedge D')                                          \\
		                       & = \lim_n \frac{1}{\lambda_1(f_\tau)^n} \min \left( L_{(f_\tau^n)_* \ord_E} (D),  L_{(f_\tau^n)_* \ord_E} (D')\right) \\
		                       & = \lim_n \min  \left( L_{\ord_E}\left(\frac{1}{\lambda_1(f_\tau)^n}  (f_\tau^n)^* (D)\right), L_{\ord_E}
		\left(\frac{1}{\lambda_1(f_\tau)^n}  (f_\tau^n)^* D'\right)\right)                                                                            \\
		                       & = \min\left( c \cdot L(D), c \cdot L(D') \right).
	\end{align}
	Since $c>0$, this proves Condition \ref{condition:L2}.
	Thus there exists a valuation $v_*$ over $K[X_0]$ such that $L = L_{v_*}$, by construction it is clear that $(f_\tau)_* L = \lambda_1(f_\tau) L$ so that we get by Corollary \ref{cor:pushforward-linear-form} that $(f_\tau)_* v_* = \lambda_1(f_\tau) v_*$.
\end{proof}

\begin{cor}
	\label{cor:dynamical-degree-algebraic}
	Let $X_0$ be a smooth affine variety over a field of characteristic zero and $\QAlb(X_0) = 0$.
	If $f_\tau : X_0 \rightarrow X_0$ is a twisted endomorphism with $\lambda_1(f_\tau)^2 > \lambda_2 (f_\tau)$, then $\lambda_1(f_\tau)$ is an algebraic number of degree $\leq \dim X_0$.
\end{cor}
\begin{proof}
	By Theorem \ref{thm:eigenvaluations}, there exists a valuation $v_*$ centered at infinity such that
		\[
			(f_\tau)_*v_*=\lambda_1(f_\tau)v_*.
		\]
		Let $\Gamma=v_*(K(X_0)^*)$ be its value group and put $\Gamma_\Q=\Gamma\otimes_\Z\Q$.
		By Abhyankar's inequality, $\Gamma_\Q$ is finite dimensional and
		\[
			\dim_\Q\Gamma_\Q\leq \dim X_0.
		\]
		The pullback by $f_\tau$ induces a $\Q$-linear endomorphism
		\[
			T:\Gamma_\Q\longrightarrow\Gamma_\Q,\qquad
			T(v_*(P))=v_*(f_\tau^*P).
		\]
		This is well defined: if $v_*(P)=v_*(Q)$, then $v_*(P/Q)=0$, hence
		\[
			v_*(f_\tau^*(P/Q))=\lambda_1(f_\tau)v_*(P/Q)=0
		\]
		by the eigenvaluation relation.
		Moreover this relation shows that $T$ acts as multiplication by $\lambda_1(f_\tau)$ on the nonzero vector space $\Gamma_\Q$.
		After choosing a $\Q$-basis of $\Gamma_\Q$, the map $T$ is represented by a matrix with rational coefficients.
Therefore $\lambda_1(f_\tau)$ is an algebraic number as it is a root of the characteristic polynomial of this matrix, and its degree is at most $\dim_\Q\Gamma_\Q\leq \dim X_0$.
	\end{proof}
\section{Algebraicity of first dynamical degree for twisted endomorphisms over affine
  varieties}\label{sec:algebraicity-dyn-degrees}
Start with the following lemma.
\begin{lemme}
	\label{lemme:induction-fibratino}
	Let $X,Y$ be quasiprojective varieties over $K$ with a dominant morphism $q : X \rightarrow Y$ and let $f_\tau, g_\tau$ be dominant twisted rational maps over $X$ and $Y$ such that $q \circ f_\tau = g_\tau \circ q$.
	Let $f_\eta$ denote the induced twisted rational map on the generic fiber of $q$.
	Suppose that $\lambda_1 (f_\tau)^2 > \lambda_2 (f_\tau)$.
	Then for $u \in \left\{ f_\eta, g_\tau \right\}$ such that $\lambda_1
		(f_\tau) = \lambda_1 (u)$ we have
	\begin{equation}
		\lambda_1^2 (u) > \lambda_2 (u).
	\end{equation}
\end{lemme}
\begin{proof}
	We have by Theorem \ref{thm:relative-degree-formula} that
	\begin{equation}
		\lambda_1 (f_\tau) = \max (\lambda_1 (u) ,\lambda_1 (v)), \quad \lambda_2 (f_\tau) =  \max \left( \lambda_2
			(u), \lambda_2 (v), \lambda_1 (u) \lambda_1 (v) \right)
	\end{equation}
	for $\left\{ u,v \right\} = \left\{ f_\eta, g_\tau \right\}$.
	Assume $\lambda_1 (f_\tau) = \lambda_1 (u)$ which means $\lambda_1 (u) \geq \lambda_1 (v)$.

	Suppose that $\lambda_2 (u) = \lambda_1 (u)^2$.
	Then
	\[
		\lambda_2(f_\tau)\geq \lambda_2(u)=\lambda_1(u)^2=\lambda_1(f_\tau)^2,
	\]
	contradicting the assumption $\lambda_1(f_\tau)^2>\lambda_2(f_\tau)$.
\end{proof}

We now prove Theorem \ref{bigthm:dyn-degrees-affine-varieties}, restated here for convenience.
\begin{thm}
	\label{thm:first-dyn-degree-algebraic}
	Let $X_0$ be a smooth affine variety of dimension $d$ over a field $K$ of characteristic zero and let $f_\tau : X_0 \rightarrow X_0$ be a dominant twisted endomorphism such that $\lambda_1 (f_\tau)^2 > \lambda_2 (f_\tau)$, then $\lambda_1 (f_\tau)$ is an algebraic number of degree $\leq (d-1)^2$ if $d\geq 3$, and $\leq d$ if $d=1,2$.
\end{thm}
\begin{proof}
	We proceed by induction on $d=\dim X_0$.
	If $d = 1$, then $\lambda_1 (f_\tau)$ is an integer.
	Assume now that $d\geq 2$ and that the result has been shown for smooth affine varieties of dimension $<d$.
	If $\QAlb(X_0) = 0$ then we have the result by Corollary \ref{cor:dynamical-degree-algebraic}.

	Otherwise, we can assume that $K$ is algebraically closed by Proposition \ref{prop:dyn-degrees-base-change} and in particular $\tau: K \rightarrow K$ is a field automorphism.

	Assume $\overline \kappa (X_0) \geq 1$.
		By Proposition \ref{prop:log-kodaira-iitaka-fibration}, for some sufficiently divisible $m$ the logarithmic Iitaka map
		\[
			q_m:X_0\dashrightarrow W_m
		\]
		has image of dimension $\overline\kappa(X_0)$ and satisfies
		\[
			q_m\circ f_\tau=g_\tau\circ q_m
		\]
		for a twisted automorphism $g_\tau$ of $W_m$ induced by a twisted projective linear automorphism of the ambient projective space of the Iitaka map.
		Hence $g_\tau$ fixes the restriction of the hyperplane class, and so $\lambda_k(g_\tau)=1$ for every $k=1,\dots,\overline\kappa(X_0)$.
		Let $\eta$ be the generic point of $W_m$, and let $f_\eta$ be the induced twisted rational map on the generic fiber of $q_m$.
By Lemma \ref{lemme:iitaka-generic-fiber-affine}, after an algebraically closed base change, $f_\eta$ can be represented by a dominant twisted endomorphism $\tilde f_\eta$ of a smooth affine variety $V_\eta$ of the same dimension over the function field of the base-changed $W_m$.
Dynamical degrees are unchanged by this base change so that $\lambda_k (f_\eta) = \lambda_k(\tilde f_\eta)$.
		By the mixed degree formula and the equality $\lambda_k(g_\tau)=1$ for all $k$, we get
		\[
			\lambda_1(f_\tau)=\lambda_1(\tilde f_\eta).
		\]
		Moreover, Lemma \ref{lemme:induction-fibratino} gives
		\[
			\lambda_1(\tilde f_\eta)^2>\lambda_2(\tilde f_\eta).
		\]
    Since $\dim V_\eta=d-\overline\kappa(X_0)<d$, the induction hypothesis applies to $\tilde f_\eta$.
		The resulting degree bound is at most the bound required in dimension $d$.

		Suppose now that $\overline \kappa (X_0) = 0$.
		By Theorem 10.1 of \cite{fujinoQUASIALBANESEMAPS2015} we have that the quasi-Albanese morphism $q : X_0 \rightarrow \QAlb (X_0)$ is dominant and with irreducible general fibers.
		Set $Q=\QAlb(X_0)$ and $r=\dim Q$.
		If $Q$ is an algebraic torus, then the result follows from the mixed degree formula, Proposition \ref{prop:absolute-dyn-degrees-algebraic-tori}, and induction on the generic fiber of $q$.
		We now show that the case $r=d$ is automatically of this kind.
		Assume $r=d$.
		Kawamata's theorem on the quasi-Albanese map, in the form recalled in \cite{fujinoQUASIALBANESEMAPS2015}, gives that $q:X_0\to Q$ is birational; moreover $q$ is proper in codimension one by \cite{lopesFootnoteTheoremKawamata2023}.
		Since $X_0$ and $Q$ are normal, it follows that $\OO(X_0)=\OO(Q)$.
		Write
		\[
			0\to T\to Q\xrightarrow{p} A\to0
		\]
		with $T$ a torus and $A$ an abelian variety.
		We recall the standard description of regular functions on such a torus torsor, see for instance \cite{brionAntiAffineAlgebraicGroups2009}.
	    If $M:=\operatorname{Hom}(T,\mathbb G_m)$ is the character group of $T$, then
		\[
			p_*\OO_Q\simeq \bigoplus_{\chi\in M} L_\chi,\qquad L_\chi\in\Pic^0(A).
		\]
		Thus
		\[
			\OO(Q)=\bigoplus_{\chi\in M} H^0(A,L_\chi).
		\]
		A nontrivial line bundle in $\Pic^0(A)$ has no nonzero global sections, so only the characters $\chi$ for which $L_\chi\simeq\OO_A$ contribute to $\OO(Q)$.
		Hence $\operatorname{trdeg}_K\OO(Q)\leq \operatorname{rank}M=\dim T$.
		On the other hand, $X_0$ is affine of dimension $d$, hence $\operatorname{trdeg}_K\OO(X_0)=d$.
		Thus $\dim T=d=\dim Q$, so $A=0$ and $Q\simeq\G_m^d$.
		Then $\OO(X_0)=\OO(Q)$ and the affineness of $X_0$ imply $X_0\simeq Q\simeq\G_m^d$.
		Therefore, after the torus case, we may assume that $r<d$.

		We have that there exists $g_\tau : Q \rightarrow Q$ such that $q \circ f_\tau = g_\tau \circ q$.
		Let $\eta$ be the generic point of $Q$, and let $f_\eta$ be the induced twisted rational map on the generic fiber of $q$.
		By the mixed degree formula,
		\[
			\lambda_1(f_\tau)=\max\{\lambda_1(f_\eta),\lambda_1(g_\tau)\}.
		\]
		The generic fiber is a smooth affine variety of dimension $d-r<d$ over $K(Q)$, while $Q$ is a quasi-abelian variety of dimension $r<d$.
		If $\lambda_1(f_\tau)=\lambda_1(f_\eta)$, then Lemma \ref{lemme:induction-fibratino} allows us to apply the induction hypothesis to $f_\eta$.
		If $\lambda_1(f_\tau)=\lambda_1(g_\tau)$, then Proposition \ref{prop:absolute-dyn-degree-quasi-abelian} shows that $\lambda_1(g_\tau)$ has degree at most $r^2\leq (d-1)^2$.
		In both cases, $\lambda_1(f_\tau)$ is algebraic of degree at most $(d-1)^2$.

	We now assume that $\overline \kappa (X_0) = - \infty$.
	Let $Q = \QAlb(X_0)$ and $q : X_0 \rightarrow Q$ be the quasi-Albanese morphism and let $V = \overline {q(X_0)}$.
	Up to translation we can assume that the neutral element $0 \in Q$ belongs to $V$.
		From Theorem \ref{thm:quotient-general-type-quasi-abelian}, let $G_V = \left\{ x \in Q: x+ V = V \right\}$.
	The twisted endomorphism $f_\tau$ yields a twisted endomorphism $g_\tau$ of $\QAlb(X_0)$ that commutes with $q$.
	Write
	the induced commutative diagram
	\begin{equation}
		\begin{tikzcd}
			X_0^{(\tau)} \ar[r, "f"] \ar[d, "q^{(\tau)}"] & X_0 \ar[d, "q"] \\
			Q^{(\tau)} \ar[r, "g"] & Q
		\end{tikzcd}
	\end{equation}
		We have that $g(G_V^{(\tau)}) \subset G_V$, indeed $g = \tau_x \circ h$ is the composition of a group homomorphism and
	a translation, i.e
	\begin{equation}
		\forall y \in Q^{(\tau)}, g(y) = h(y) + x.
	\end{equation}
		Since $g(V^{(\tau)}) \subset V$ and $0 \in V$ we have that $x \in V$, now if $k \in G_V^{(\tau)}$ then for any
	$v \in V^{(\tau)}$ we have
	\begin{equation}
		g(v+k) = h(v+k) + x = h(k) + g(v).
	\end{equation}
		Since $g (V^{(\tau)}) = V$ we have that $h(G_V^{(\tau)}) \subset G_V$ and in particular this implies that $g$ descends
		to a dominant map
		\begin{equation}
			g: V^{(\tau)} / (G_V^0)^{(\tau)} \rightarrow V / G_V^0.
		\end{equation}
		This yields a twisted endomorphism $g_\tau: V / G_V^0 \rightarrow V / G_V^0$ and we still denote by $q : X_0 \rightarrow V / G_V^0$ the projection.
		We distinguish several cases.

		\textbf{$\dim V / G_V^0 \geq 1$}.-- By Theorem \ref{thm:quotient-general-type-quasi-abelian}, $V / G_V^0$ is of
		log general type.
		Applying Proposition \ref{prop:log-kodaira-iitaka-fibration} to $V/G_V^0$, its logarithmic Iitaka map is birational and semiconjugates $g_\tau$ to a twisted projective linear automorphism.
		By birational invariance of dynamical degrees, $\lambda_1(g_\tau)=1$.
	By Theorem \ref{thm:relative-degree-formula} we have that $\lambda_1 (f_\tau) = \lambda_1 (f_\eta)$ and $f_\eta$ is a twisted endomorphism of the general fiber of $q$ which is an affine variety of lower dimension so the result follows by induction using Lemma \ref{lemme:induction-fibratino}.

		\textbf{$V / G_V^0 = 0$}.-- Since we have translated $V$ so that $0\in V$, this means that $V=G_V^0$.
		But the image of the quasi-Albanese map generates $Q=\QAlb(X_0)$ as a quasi-abelian variety.
		Hence $G_V^0=Q$, so $V=Q$ and $q:X_0\to Q$ is dominant.
	Since $\overline \kappa (X_0) = - \infty$ we cannot have that $q : X_0 \rightarrow Q$ is generically finite, otherwise we would get that $\overline \kappa (X_0) \geq \overline \kappa (Q) = 0$ and this is a contradiction.
	Thus $r:=\dim Q<d$, so $r\leq d-1$.
	Let $\eta$ be the generic point of $Q$, and let $f_\eta$ be the induced twisted endomorphism on the generic fiber of $q$.
	By the mixed degree formula,
	\[
		\lambda_1(f_\tau)=\max\{\lambda_1(f_\eta),\lambda_1(g_\tau)\}.
	\]
	If the maximum is attained by $\lambda_1(f_\eta)$, the result follows by induction using Lemma \ref{lemme:induction-fibratino}.
	If the maximum is attained by $\lambda_1(g_\tau)$, then Proposition \ref{prop:absolute-dyn-degree-quasi-abelian} gives degree at most $r^2\leq (d-1)^2$.
	This proves the result.
\end{proof}

We believe that the optimal bound for the degree of the algebraic number $\lambda_1$ is the dimension of $X_0$.
We see that the only case this bound is not optimal is if our dynamical system fibers through an abelian variety of high enough dimension with non trivial dynamics.

\subsection*{Acknowledgements}
Junyi Xie is supported by the NSFC Grant No. 12271007.
Junyi Xie thanks the AI assistant Xiaozhua for assistance during the writing process, and Liang Xiao and Shuai Chen for providing the AI-agent environment and for resolving related technical issues.
Marc Abboud acknowledges support by the Swiss National Science Foundation Grant “Birational transformations of higher dimensional varieties” 200020-214999.

\bibliographystyle{alpha}
\bibliography{biblio}
\end{document}